\documentclass[11pt]{amsart}

\usepackage{amstext}
\usepackage{amsmath}
\usepackage{amsfonts}
\usepackage{latexsym}

\usepackage{amssymb}
\usepackage[all]{xy}

\usepackage{hyperref}
\theoremstyle{plain}

\newtheorem{thm}{Theorem}[section]

\newtheorem{pro}[thm]{Proposition}

\newtheorem{lem}[thm]{Lemma}

\newtheorem{cor}[thm]{Corollary}
\newtheorem{con}[thm]{Conjecture}

\theoremstyle{definition}

\newtheorem{dfn}[thm]{Definition}

\newtheorem{nt}[thm]{Notation}

\newtheorem{rem}[thm]{Remark}

\theoremstyle{remark}

\newcommand{\Z}{\mathbb{Z}}
\newcommand{\N}{\mathbb{N}}
\newcommand{\C}{\mathbb{C}}
\newcommand{\R}{\mathbb{R}}
\newcommand{\Q}{\mathbb{Q}}
\newcommand{\PS}{\mathbb{P}}

\newcommand{\OO}{\mathcal{O}}

\newcommand{\mcal}{\mathcal}

\DeclareMathOperator{\inte}{int}

\DeclareMathOperator{\rk}{rk}
\DeclareMathOperator{\Coker}{Coker}

\DeclareMathOperator{\Supp}{Supp}

\DeclareMathOperator{\GL}{GL}
\DeclareMathOperator{\Hom}{Hom}

\DeclareMathOperator{\Aut}{Aut}

\DeclareMathOperator{\Pic}{Pic} 

\DeclareMathOperator{\Div}{Div}

\DeclareMathOperator{\NEb}{\overline{\mathrm{NE}}}

\DeclareMathOperator{\Nef}{Nef}

\newcommand{\id}{{\rm id}}

\newcommand{\E}{{\mathcal E}}

\newcommand\sE{{\mathcal E}}
\newcommand\sF{{\mathcal F}}

\newcommand\sI{{\mathcal I}}

\newcommand\sL{{\mathcal L}}
\newcommand\sO{{\mathcal O}}
\newcommand\sS{{\mathcal S}}

\newcommand\bP{{\mathbb P}}

\title{Nef line bundles on Calabi-Yau threefolds, I} 

\author{Vladimir Lazi\'c}
\address{Fachrichtung Mathematik, Campus, Geb\"aude E2.4, Universit\"at des Saarlandes, 66123 Saarbr\"ucken, Germany}
\email{lazic@math.uni-sb.de}

\author{Keiji Oguiso}
\address{Graduate School of Mathematical Sciences, University of Tokyo, Komaba, Meguro, Tokyo, 153-8914, Japan and Korea Institute for Advanced Study, Hoegiro 87, Seoul, 130-722, Korea}
\email{oguiso@ms.u-tokyo.ac.jp}

\author{Thomas Peternell}
\address{Mathematisches Institut, Universit\"at Bayreuth, 95440 Bayreuth, Germany}
\email{thomas.peternell@uni-bayreuth.de}

\thanks{All authors were partially supported by the DFG-Forschergruppe 790 ``Classification of Algebraic Surfaces and Compact Complex Manifolds". Lazi\'c was supported by the DFG-Emmy-Noether-Nachwuchsgruppe ``Gute Strukturen in der h\"oherdimensionalen birationalen Geometrie". Oguiso was supported by JSPS Grant-in-Aid (S) No 25220701, JSPS Grant-in-Aid (S) No 22224001, JSPS Grant-in-Aid (B) No 22340009, and by KIAS Scholar Program. We would like to thank J.\ Chen, E.\ Floris, R.\ Svaldi and L.\ Tasin for useful conversations related to this work. Last but not least, we thank the referees whose many comments were extremely valuable.}

\begin{document}

\begin{abstract}
{We prove that a nef line bundle $\sL$ with $c_1(\sL)^2 \ne 0$ on a Calabi-Yau threefold $X$ with Picard number $2$ and with $c_3(X) \ne 0$ is semiample, that is, some multiple of $\mathcal L$ is generated by global sections.}
\end{abstract}

\maketitle
\setcounter{tocdepth}{1}
\tableofcontents

\section{Introduction}

The following is a standard conjecture in the theory of simply connected Ricci-flat compact K\"ahler manifolds.

\begin{con}\label{con:1}
Let $X$ be a simply connected compact K\"ahler manifold with $c_1(X) \equiv 0$. Let $\mathcal L$ be a nef line bundle on $X$. Then $\mathcal L$ is semiample, that is, there exists a positive integer $m$ such that $\mathcal L^{\otimes m}$ is generated by global sections.
\end{con}

Recall that a line bundle $\mathcal L$ on a projective manifold $X$ of dimension $n$ is nef if $c_1(\mathcal L ) \cdot C \geq 0$ for every irreducible curve $C$ on $X$; in the K\"ahler setting, $\mathcal L$ being nef means that $c_1(\mathcal L) $ is in the closure of the K\"ahler cone. In this paper we consider the case of dimension three, in which setting $X$ is automatically projective. Conjecture \ref{con:1}, which should be seen as a stronger form of log abundance on Calabi-Yau manifolds, is rather mysterious unless $c_1(\sL)^{\dim X} > 0$, and almost nothing is known when $\dim X>2$.

Our results towards Conjecture \ref{con:1} are the following. The first result describes a general strategy towards Conjecture \ref{con:1}. 

\begin{thm} \label{thm:introample} 
Let $X$ be a three-dimensional simply connected projective manifold with $c_1(X) \equiv 0$ and $c_3(X) \ne 0$. Let $\mathcal L$ be a nef line bundle on $X$ with $c_1(\mathcal L)^2 \not\equiv 0$. Let $G$ be a smooth ample divisor on $X$. Assume that there exists a very ample divisor $H$ on $X$ and a positive integer $m$ such that for general $D \in | H |$ (so that $G+D$ has simple normal crossings) we have: 
\begin{enumerate} 
\item[(i)]  the locally free sheaf $\Omega^1_X\big(\log(D+G)\big) \otimes \mathcal L^{\otimes m}$ is nef, and
\item[(ii)] the line bundle  $\mathcal L|_{D+G} $ is ample. 
\end{enumerate} 
Then $\mathcal L$ is semiample. 
\end{thm} 

This result is Theorem \ref{ample} and is a special case of Theorem \ref{topology} below. Here, recall that a locally free sheaf $\mathcal E$ is nef if the line bundle $\OO_{\PS(\mathcal E)}(1)$ is nef. Note that the condition that $\mathcal L|_{G+D}$ is ample is equivalent to the ampleness of $\mathcal L|_{G}$ and of  $\mathcal L|_{D}$
and therefore by Proposition \ref{pro:trivialsurface},
to saying that there is no irreducible surface $S$ on $X$ such that $c_1(\mathcal L) \cdot S = 0$. 

As a consequence of Theorem \ref{thm:introample} we obtain the following:
 
\begin{thm}  \label{MT2} 
Let $X$ be a three-dimensional simply connected projective manifold with $c_1(X) \equiv 0$ and with Picard number $2$. Suppose further that $c_3(X) \ne 0$. Then any nef line bundle $\mathcal L$ on $X$ with  $c_1(\mathcal L)^2 \not \equiv 0$ is semiample. 
\end{thm} 

Theorem \ref{MT2} was claimed in \cite{Wi94}, even in a more general setting, but we were unable to follow the proof; cp.\ Theorem \ref{ample}. However, the general ideas of  \cite{Wi94} are important. 

The assumption that $\rho(X) = 2$  in Theorem \ref{MT2} is essential in the proof, but we hope that the strategy and some of the methods can be useful also in the general case. We also emphasize that there is a wealth of interesting families of Calabi-Yau threefolds of Picard number $\rho(X) =2$; we refer to \cite{LP13, Og14} and the references given therein. Notice also that if $\sL$ is a semiample line bundle on a Calabi-Yau threefold with $c_1(\sL)^2\not\equiv 0$  and $c_1(\sL) \cdot c_2(X) = 0$, then $\rho(X) \geq 3$ by \cite{Og93}. Therefore, in order to remove the assumption $c_3(X) \ne 0$ in Theorem \ref{MT2}, one should prove the non-existence of a nef line bundle $\mathcal L$ such that 
$c_1(\mathcal L)^2 \not \equiv 0 $ and $c_1(\mathcal L) \cdot c_2(X) = 0$ rather than proving abundance for $\mathcal L$. 

\medskip

The semiampleness of a nef line bundle $\sL$ on a simply connected projective threefold $X$ with $c_1(X) \equiv 0$ is obvious if $c_1(\sL) \equiv 0$ (in which case $\mathcal L \simeq \OO_X$), and when $\mathcal{L}$ is big, it is a consequence of the basepoint free theorem \cite{Sho85,Kaw85b}.  However, if there is no bigness assumption, the existence of sections is notoriously difficult. Remarkably, as a consequence of log abundance for threefolds, in order to show semiampleness of $\mathcal L$, by \cite{Og93,KMM94,Kaw92}, it suffices to show that 
$$H^0(X,\mathcal L^{\otimes m}) \ne 0\quad\text{for some positive }m.$$ 

Our results are also related to the Cone conjecture of Morrison and Kawamata. As we discuss in Section \ref{sec:prelim}, a consequence of Kawamata's formulation of the Cone conjecture is the following structural prediction.

\begin{con}\label{con:2}
Let $X$ be a $\Q$-factorial projective variety with klt singularities and with numerically trivial canonical class $K_X$. Then  $\Nef(X)^+=\Nef(X)^e$.
\end{con}

Here, $\Nef(X)^+$ and $\Nef(X)^e$ are the parts of the nef cone of $X$ spanned by rational, respectively effective, classes. It seems to have been unknown thus far whether one of these cones is a subset of the other, unless $X$ is an abelian variety or a hyperk\"ahler manifold \cite{Bou04}; 
see also \cite{CO15} for results on special Calabi-Yau manifolds in any dimension. 

Theorem \ref{thm:nef+} below gives a short proof that in the most general setting we have $\Nef(X)^e\subseteq\Nef(X)^+$. Thus, Theorem \ref{MT2} together with Theorem \ref{thm:nef+} can be restated as follows.

\begin{cor}
Let $X$ be a three-dimensional simply connected projective manifold with $c_1(X) \equiv 0$ and with Picard number $2$. Assume that $c_3(X) \ne 0$ and that $X$ does not carry a non-trivial line bundle $\sL$ with $c_1(\sL)^2 \equiv 0$ and $c_1(\sL) \cdot c_2(X) = 0$. Then Conjecture \ref{con:2} holds: $\Nef(X)^+=\Nef(X)^e$.
\end{cor}

In fact, by Theorem \ref{MT2} and the discussion below it, every nef line bundle on $X$ is semiample. 

Another application concerns the existence of rational curves on Calabi-Yau threefolds. It is known that rational curves exist on Calabi-Yau threefolds with Picard number at least $14$, cf.\ \cite[Theorem]{HBW92}, but almost nothing is known for smaller Picard number apart from \cite{DF14}. Using \cite{Pe91,Og93}, in Section \ref{section4} we deduce the following:

\begin{cor}\label{cor:rational}
Let $X$ be a three-dimensional simply connected projective manifold with $c_1(X) \equiv 0$, $c_3(X) \neq 0$ and with Picard number $2$. Assume that not both boundary rays of the cone $\Nef(X)$ are irrational. Then $X$ has a rational curve.
\end{cor}

We spend a few words on the method of the proof of Theorem \ref{thm:introample} and  Theorem \ref{MT2}, given in Sections \ref{sec:ample}, \ref{sec:nu=2} and \ref{sec:semicriterion2}. For Theorem \ref{thm:introample}, arguing by contradiction and possibly replacing $\mathcal L$ by a multiple, we first observe that
$$ H^0(X,\Omega^1_X \otimes \sL^{\otimes m}) = 0\quad\text{for all }m, $$
cf.\ \cite[3.1]{Wi94}. Our method is then to study the cohomology of logarithmic differentials with poles along a smooth very ample divisor $D$ and with values in $\mathcal L^{\otimes m}$. As the final outcome, we show the crucial vanishing 
$$ H^2(X,\Omega^1_X \otimes \sL^{\otimes m}) = 0\quad\text{for large }m.$$
Then it is not difficult to see that 
\begin{equation}\label{eq:strange}
H^q(X,\Omega^p_X \otimes \sL^{\otimes m}) = 0 \quad\text{for all $p$ and $q$, and all }m \text{ with }| m | \gg 0. 
\end{equation} 
Using the Hirzebruch-Riemann-Roch, we deduce $c_3(X) = 0$. We conjecture that the vanishings \eqref{eq:strange} can never happen. 

For Theorem \ref{MT2}, the main point is to establish the nefness of logarithmic bundles 
$$ \Omega^1_X(\log D) \otimes \mathcal L^{\otimes m} $$ 
for a carefully chosen very ample divisor $D$, and for this choice of $D$ the assumption $\rho(X) = 2$ is needed, see Proposition \ref{nef}.  

In Section \ref{sec:semicriterion2} we generalise Theorem \ref{thm:introample} to the case when there exist (necessarily finitely many) surfaces $S$ such that $c_1(\mathcal L)\cdot S=0$.  

\section{Preliminaries}\label{sec:prelim}

Throughout the paper we work over the field of complex numbers $\C$. For a variety $X$, $\Pic(X)$ is the Picard group of $X$ and $N^1(X)$ is the N\'eron-Severi group of $X$. The Picard number of $X$ is $\rho(X)=\rk N^1(X)$. If $\sL$ is a nef Cartier divisor on $X$, the \emph{numerical dimension of $\sL$} is 
$$\nu(X,\sL)=\max\{k\in\N\mid \sL^k\not\equiv0\}.$$
If $D =\sum_{i=1}^k D_i$ is a reduced divisor with simple normal crossings on a smooth variety $X$, recall that we have the locally free sheaf of logarithmic differentials $\Omega^1_X(\log D)$ together with the exact \emph{residue sequences}
\begin{equation}\label{eq:residue1}
0 \to \Omega^1_X \to \Omega^1_X (\log D) \to \bigoplus_{i=1}^k\sO_{D_i} \to0,
\end{equation}
\begin{equation}\label{eq:residue2}
0 \to \Omega^1_X\big(\log (D-D_k)\big) \to   \Omega^1_X(\log D) \to \sO_{D_k} \to 0,
\end{equation}
and
\begin{align}\label{eq:residue3}
0 \to \Omega^1_X(\log D)(-D_k) &\to  \Omega^1_X\big(\log(D-D_k)\big) \\
&\to \Omega^1_{D_k}\big(\log(D-D_k)|_{D_k}\big) \to 0,\notag
\end{align}
cf.\ \cite[\S 2]{EV92}.

\subsection{Calabi-Yau threefolds}

A Calabi-Yau manifold is by definition a compact K\"ahler manifold $X$ which is simply connected and has trivial canonical bundle $\omega_X \simeq \sO_X$. A three-dimensional Calabi-Yau manifold is simply called \emph{Calabi-Yau threefold}, and thus a Calabi-Yau threefold in this paper is always smooth and projective. We notice that $h^1(X,\OO_X)=h^2(X,\OO_X)=0$ and hence $\Pic(X)\simeq N^1(X)$. 

If $\mathcal{E}$ is a vector bundle on a Calabi-Yau threefold $X$, then by the Hir\-ze\-bruch-Riemann-Roch theorem \cite[p.\ 432]{Har77} we have
\begin{equation}\label{eq:RR1}
\chi(X,\mathcal E)=\frac{1}{12}c_1(\mathcal{E})c_2(X)+\frac16\big(c_1(\mathcal{E})^3-3c_1(\mathcal{E})c_2(\mathcal{E})+3c_3(\mathcal{E})\big).
\end{equation}
In particular, for a Cartier divisor $L$ on $X$ this gives
\begin{equation}\label{eq:RR2}
\chi(X,L)=\frac16L^3+\frac{1}{12}L\cdot c_2(X).
\end{equation}
For a vector bundle $\mathcal E$ of rank $3$ and a line bundle $\mathcal L$ on $X$, \cite[Example 3.2.2]{Ful98} gives
\begin{align}\label{eq:chern}
c_1(\mathcal E\otimes\mathcal L)&=3c_1(\mathcal{L})+c_1(\mathcal{E}),\notag\\ c_2(\mathcal E\otimes\mathcal L) &=3c_1(\mathcal{L})^2+2c_1(\mathcal{E})c_1(\mathcal L)+c_2(\mathcal{E}),\\
c_3(\mathcal E\otimes\mathcal L) &=c_1(\mathcal{L})^3+c_1(\mathcal{E})c_1(\mathcal L)^2+c_2(\mathcal{E})c_1(\mathcal L)+c_3(\mathcal{E}).\notag
\end{align}

Given a locally free sheaf $\sE$, we denote its $k$-th Segre class by $s_k(\sE)$. The sign is determined in such a way that if $\sE$ has rank $r$ and $\dim X = n$, then 
$$ s_r(\sE) = c_1\big(\sO_{\bP(\sE)}(1)\big)^{n+r-1}.$$
Here $\bP(\sE) $ is defined by taking hyperplanes; in particular, $s_k(\sE) $ differs in sign from the definition in \cite[Chapter 3]{Ful98} in the case when $k$ is odd. 

\begin{pro} \label{pro:Miyaoka}
If $X$ is a Calabi-Yau threefold, then $c_2(X) \in \NEb(X)\backslash\{0\}$.
\end{pro}

\begin{proof}
We have $c_2(X) \in \NEb(X)$ by \cite{Mi87} and $c_2(X) \ne 0$ by Yau's theorem, see for example, \cite[IV.4.15]{Kob87}. 
\end{proof}

\begin{pro}\label{prop1} 
Let $X$ be a Calabi-Yau threefold, let $L$ be a nef Cartier divisor on $X$ which is not semiample, and let $D_1$ and $D_2$ be prime divisors on $X$ such that $D_1+D_2$ has simple normal crossings. Denote 
$$\mathcal E_m=\Omega^1_X(\log D_1) \otimes \OO_X(mL)$$
and 
$$\mathcal E_m'=\Omega^1_X\big(\log (D_1+D_2)\big) \otimes \OO_X(mL).$$
Then 
\begin{enumerate} 
\item[(i)] $\kappa (X,L) = {-} \infty$,
\item[(ii)] $L^3 = L \cdot c_2(X) = 0$,
\item[(iii)] $\chi\big(X,\OO_X(mL)\big)=0$ for every $m$,
\item[(iv)] $\chi\big(X,\Omega^1_X\otimes \OO_X(mL)\big)= - \frac{1}{2}c_3(X)$ for every $m$,
\item[(v)] for every $m$ we have
$$ \chi(X,\mathcal E_m) = \frac12m^2L^2\cdot D -\frac{1}{2}mL\cdot D^2 + \frac{1}{6}D^3 + \frac{1}{12} D \cdot c_2(X)  -\frac{1}{2} c_3(X),$$
\item[(vi)] for every $m$ we have 
$$s_3(\mathcal E_m')=10m^2L^2 \cdot (D_1+D_2) +5mL\cdot D_1\cdot D_2 - (D_1+D_2) \cdot c_2(X) - c_3(X).$$ 
\end{enumerate}
\end{pro}

\begin{proof}
The statement (i) is \cite{Og93}, or it can be viewed as a consequence of log abundance for threefolds \cite{KMM94}. Statement (ii) is the basepoint free theorem together with \cite[2.7]{Og93}. From (ii) and \eqref{eq:RR2} we obtain (iii).

Set $\mathcal F_m=\Omega^1_X \otimes \OO_X(mL)$.  Using \eqref{eq:chern} and (ii), we calculate:
$$c_1(\mathcal F_m)=3mL,\quad c_2(\mathcal F_m)=3m^2L^2+c_2(X),\quad c_3(\mathcal F_m)= - c_3(X),$$
which gives (iv) by \eqref{eq:RR1} and (ii). 

Since $c_j(\OO_D)=D^j$ for $1 \leq j \leq 3$,  from \eqref{eq:residue1} we obtain
\begin{align*}
c_1\big(\Omega_X^1(\log D)\big)&=D,\quad c_2\big(\Omega_X^1(\log D)\big)=c_2(X)  + D^2,\\ 
c_3\big(\Omega_X^1(\log D)\big)&=c_2(X)\cdot D -c_3(X)+ D^3.
\end{align*}
Using \eqref{eq:chern} and (ii), we calculate:
\begin{align*}
c_1(\mathcal E_m)&=3mL+D,\quad c_2(\mathcal E_m)=3m^2L^2+2mL\cdot D+c_2(X) + D^2,\\
c_3(\mathcal E_m)&=m^2L^2\cdot D+m L \cdot D^2+c_2(X)\cdot D-c_3(X) + D^3, 
\end{align*}
which gives (v) by \eqref{eq:RR1} and (ii). 

Finally, similarly as above we obtain
\begin{align*}
c_1\big(\Omega_X^1(\log (D_1+D_2))\big)&=D_1+D_2,\\
c_2\big(\Omega_X^1(\log (D_1+D_2))\big)&=c_2(X)  + D_1^2 + D_2^2 + D_1\cdot D_2,\\ 
c_3\big(\Omega_X^1(\log (D_1+D_2))\big)&=c_2(X)\cdot (D_1+D_2) -c_3(X)\\
&+ D_1^3+D_2^3+D_1^2\cdot D_2+D_1\cdot D_2^2,
\end{align*}
and hence:
\begin{align*}
c_1(\mathcal E_m')&=3mL+D_1+D_2,\\
c_2(\mathcal E_m')&=3m^2L^2+2mL\cdot (D_1+D_2)+c_2(X) + D_1^2 + D_2^2 + D_1\cdot D_2,\\
c_3(\mathcal E_m')&=m^2L^2\cdot (D_1+D_2)+m L \cdot (D_1^2+D_2^2+D_1\cdot D_2)\\
&+c_2(X)\cdot (D_1+D_2)-c_3(X) + D_1^3 + D_2^3 + D_1^2\cdot D_2 + D_1\cdot D_2^2. 
\end{align*}
Now (vi) follows from the formulas above, from (ii) and from the formula 
$$s_3(\mathcal E_m)= c_1(\mathcal E_m)^3 - 2c_1(\mathcal E_m) \cdot c_2(\mathcal E_m) + c_3(\mathcal E_m).$$ 
This finishes the proof.
\end{proof}

\begin{lem} \label{lem1}  
Let $X$ be a Calabi-Yau threefold with $\rho(X) = 2$, and assume that there exists a nef Cartier divisor $L$ on $X$ which is not semiample. If $S \subseteq  X$ is an irreducible surface, then $S \cdot c_2(X) > 0$. 
\end{lem} 

\begin{proof} 
Fix an ample effective divisor $H$ on $X$.  Since $\rho(X)=2$, we can write
$$ S = \alpha L + \beta H \quad\text{with }\alpha, \beta\in\Q.$$
Notice that $\beta > 0$, since otherwise $\alpha L $ is effective, that is $\kappa(X,L)\geq0$, which contradicts Proposition \ref{prop1}(i). Hence
$$ S \cdot c_2(X) = \beta H \cdot c_2(X) > 0,$$
by Proposition \ref{prop1}(ii) and Proposition \ref{pro:Miyaoka}.
\end{proof} 

\begin{lem}\label{lem:surface}
Let $X$ be a Calabi-Yau threefold with $\rho(X) = 2$, and assume that there exists a nef Cartier divisor $L$ on $X$ with $\nu(X,L)=2$ which is not semiample. Then for every surface $S$ on $X$ we have $L^2\cdot S>0$.
\end{lem}

\begin{proof}
Assume that $S$ is a surface on $X$ with $L^2\cdot S=0$. Since $L$ is not semiample, $S$ is not proportional to $L$ by Proposition \ref{prop1}(i). Therefore, as $\rho(X)=2$, the divisors $L$ and $S$ form a basis of $N^1(X)_\R$.  Since also $L^3 = 0$, the $1$-cycle $L^2$ is orthogonal to $L$ and $S$, hence $L^2\equiv 0$, a contradiction to
the assumption $\nu(X,L) = 2$. 
\end{proof}

\begin{lem} \label{generic} 
Let $X$ be a Calabi-Yau threefold with $\rho(X) = 2$, and assume that there exists a nef Cartier divisor $L$ on $X$ with $\nu(X,L)=2$ which is not semiample. \begin{enumerate}
\item[(a)] There exist only countably many curves $C_i$ on $X$ such that $L \cdot C_i = 0$.
\item[(b)] If $A$ is a basepoint free ample divisor and if $D$ is a very general member of the linear system $| A | $, then $C_i \not \subseteq D$ for all $i$, and $L|_D$ is ample.
\end{enumerate}
\end{lem} 

\begin{proof} 
Assume (a) does not hold. Then, since the Chow scheme of $X$ has only countably many components, there exists a one-dimensional family $(C_t)_{t \in T}$ of generically irreducible curves such that 
\begin{equation}\label{eq:78}
L \cdot C_t = 0\quad\text{for all }t\in T. 
\end{equation}
Let $S$ be the irreducible surface covered by the curves $C_t$. Then $L|_S$ is a nef divisor which is not big: 
otherwise, by Kodaira's trick we could write $L|_S\sim_\Q H+E$ for ample, respectively effective, $\Q$-divisors $H$ and $E$ on $S$. But then \eqref{eq:78} implies $E\cdot C_t<0$ for all $t\in T$, hence $C_t\subseteq\Supp E$, a contradiction.
Therefore $(L|_S)^2 = 0$, a contradiction with Lemma \ref{lem:surface}. This shows (a). 

For (b), if $D \in | A | $ is general, then clearly $C_i \not \subseteq D$ by (a), and $L|_D$ is ample by the Nakai-Moishezon criterion due to $L^2 \cdot D > 0$ by Lemma \ref{lem:surface}.
\end{proof} 

\subsection{Nef reduction} 

We need the notion of the nef reduction of a nef divisor. The following is the main result of \cite{BCE+}.

\begin{thm}\label{thm:nefreduction}
Let $L$ be a nef divisor on a normal projective variety $X$. Then there exists an almost holomorphic dominant rational map $f\colon X\dashrightarrow Y$ with connected fibres, called \emph{nef reduction of $L$}, such that:
\begin{enumerate}
\item[(i)] $L$ is numerically trivial on all compact fibres $F$ of $f$ with $\dim F=\dim X-\dim Y$,
\item[(ii)] for every very general point $x\in X$ and every irreducible curve $C$ on $X$ passing through $x$ and not contracted by $f$, we have $L\cdot C>0$.
\end{enumerate}
The map $f$ is unique up to birational equivalence of $Y$.
\end{thm}

The following is an immediate consequence.

\begin{thm} \label{thm:applnefred} 
Let $X$ be a Calabi-Yau threefold and let $L$ be a nef divisor on $X$. Let $(C_t)_{t \in T}$ be a family of generically irreducible curves (in the Chow scheme) with $T$ an irreducible compact variety of dimension $2$ such that $L \cdot C_t = 0$ for all $t$. Let $ Y = \bigcup\nolimits_{t \in T} C_t$.
\begin{enumerate}
\item[(i)] If $Y$ is a surface, then $L \cdot Y = 0$.
\item[(ii)] Suppose additionally that $\nu(X,L) = 2$. If $Y=X$, then $L$ is semiample.
\end{enumerate} 
\end{thm}

\begin{proof}
For (i), let $\eta\colon\widetilde Y \to Y$ be the normalisation. By Theorem \ref{thm:nefreduction} applied to $\widetilde Y$, the divisor $\eta^*(L|_Y) \equiv 0$, hence $L|_Y \equiv 0$, so that $L \cdot Y = 0$.

For (ii), let $\phi\colon X \dasharrow S$ be a nef reduction of $L$; in particular, $\dim \phi(C_t) = 0 $ for each $t$ such that $C_t$ is not contained in the indeterminacy locus of $\phi$. Clearly $\dim S < 3$. If $\dim S = 0$, then $L\equiv 0$, hence $L\sim_\Q 0$. If $\dim S = 1$, then $\phi$ is holomorphic and by \cite[Theorem 1.3]{Leh15} we have $L \sim_\Q \phi^*A$, where $A$ is ample on $S \simeq \PS^1$. 

Finally, if $\dim S = 2$, let $q\colon\mathcal C \to \widetilde T$ be the normalised graph of the family $(C_t)_{t \in T}$ with projection $p\colon \mathcal C \to X$. We may even assume $\widetilde T$ to be smooth, but $\mathcal C$ is only normal. Since $\dim S = 2$ and $\phi$ is almost holomorphic, there is a unique curve $C_t$ through the general point $x \in X$. Hence $p$ is birational. Let $\pi\colon \widehat {\mathcal C} \to  {\mathcal C}$ be a desingularisation; hence $p \circ \pi\colon \widehat {\mathcal C} \to X$ is a birational map between projective manifolds. Consequently,
$$ H^1\big(\widehat { \mathcal C}, \OO_{\widehat {\mathcal C}}\big) \simeq H^1(X,\OO_X) = 0.$$ 
Then the Leray spectral for $\pi$ yields $H^1(\mathcal C,\OO_{\mathcal C}) = 0$. By \cite[Theorem 1.2]{Leh15}, we then have $p^*L \sim_\Q q^*A$, where $A$ is a nef Cartier divisor on the surface $\widetilde T$. Since $\nu(X,L) = 2$, the divisor $A$ is big, hence $\kappa(X,L) = 2$ and $L$ is semiample by Proposition \ref{prop1}. 
\end{proof} 

\begin{pro} \label{pro:trivialsurface} 
Let $X$ be a Calabi-Yau threefold and let $L$ be a nef divisor on $X$ which is not semiample such that $\nu(X,L) = 2$. Then the following are equivalent.
\begin{enumerate}
\item[(i)] There exists a very ample divisor $H$ on $X$ such that $L|_D$ is ample for a general $D \in | H |$.
\item[(ii)] For every irreducible surface $S \subseteq X$ we have $L \cdot S \not\equiv 0$. 
\end{enumerate}
\end{pro} 

\begin{proof} 
Clearly (i) implies (ii). Conversely, assume that for every very ample divisor $H$ on $X$, $L|_D$ is not ample for general $D \in | H |$. Since $L^2 \cdot D \ne 0$, by the Nakai-Moishezon criterion there exists a curve $C \subseteq D$ such that $L \cdot C = 0$. Varying $D$, we obtain an at least $2$-dimensional family of (generically) irreducible curves $(C_t)_{t \in T}$  such that $L \cdot C_t = 0$ for all $t$. If the family $(C_t)$ covers $X$, then $L$ is semiample by Theorem \ref{thm:applnefred}(ii), a contradiction. If the family $(C_t)$ covers a surface $S$, then $L \cdot S = 0$ by Theorem \ref{thm:applnefred}(i).
\end{proof} 

\subsection{Positivity of locally free sheaves}

Recall that a locally free sheaf $\mathcal E$ on a variety $X$ is nef, respectively big, if the line bundle $\sO_{\bP(\mathcal E)}(1) $ is nef, respectively big. We first gather several properties of nef sheaves that we use in the paper.

\begin{lem}\label{lem:nefVB}
Let $X$ be a projective manifold.
\begin{enumerate}
\item[(a)] If $\mathcal E$ is a nef locally free sheaf on $X$, then every locally free quotient sheaf of $\mathcal E$ is nef.
\item[(b)] Let $0\to \mathcal E\to \mathcal F\to \mathcal G\to0$ be an exact sequence of locally free sheaves. If $\mathcal E$ and $\mathcal G$ are nef, then $\mathcal F$ is nef.
\item[(c)] Assume that $X$ is a curve. A locally free sheaf $\mathcal E$ on $X$ is nef if and only if all locally free quotients have non-negative degree, and is ample if and only if all locally free quotients have positive degree. 
A semistable locally free sheaf $\mathcal E$ on $X$ is nef if and only if it has a non-negative degree, and is ample if and only if it has a positive degree.
\item[(d)] Assume that $X$ is a smooth curve, and let $\mathcal E\to \mathcal F\to \mathcal G \to0$ be an exact sequence, where $\mathcal E$, $\mathcal F$ are locally free sheaves, and $\mathcal E$ and $\mathcal G /{\rm torsion} $ are nef. Then $\mathcal F$ is nef.
\item[(e)] Assume that $X$ is a smooth curve and let $\mathcal E$ be a locally free sheaf on $X$ which is generically globally generated. Then $\mathcal E$ is nef.
\item[(f)] Assume that $X$ is a (not necessarily smooth) curve, let $\mathcal E$ be a coherent sheaf on $X$ which is generically globally generated, and assume there is a generically surjective morphism $\mathcal{E}\to\mathcal Q$ to a locally free sheaf $\mathcal Q$ on $X$. Then $\mathcal Q$ is nef.
\end{enumerate}
\end{lem}

\begin{proof}
For (a), (b) and (c), see for instance \cite[Proposition 6.1.2, Theorem 6.2.12, and the proof of Theorem 6.4.15]{Laz04}. 

For (d), by replacing $\mathcal E$ by the image of $\mathcal E$ in $\mathcal F$, by (a) we may assume that the exact sequence is also exact on the left. Let $\mathcal E'$ be the saturation of $\mathcal E$ in $\mathcal F$, so that we have the exact sequence 
$$0\to \mathcal E'\to \mathcal F\to \mathcal G /{\rm torsion} \to0.$$
Since $\mathcal G/{\rm torsion} $ is locally free,  it suffices by (b) to show that $\mathcal E'$ is nef. Assuming otherwise, by (c) there exists a locally free quotient $\mathcal E'\to \mathcal Q' \to 0 $ with $\deg \mathcal Q'<0$, and it induces a quotient $\mathcal E\to \mathcal Q \to 0 $ with $\mathcal Q\subseteq \mathcal Q'$. In particular, $\deg \mathcal Q\leq\deg \mathcal Q'$, a contradiction since $\mathcal E$ is nef.

To prove (e) consider the subsheaf $\mathcal E' \subseteq \mathcal E$ generated by the global sections of $\mathcal E$. The quotient $\mathcal E/\mathcal E'$ is torsion, and we conclude by (d). 

Finally, for (f) let $\nu\colon\widetilde X\to X$ be the normalisation. We first note that the generically surjective morphism $\nu^*\mathcal E\to\nu^*\mathcal Q$ factors through the locally free sheaf $\mathcal E':=\nu^*\mathcal E/{\rm torsion}$, and consider the induced generically surjective morphism $\varphi\colon\mathcal E'\to\nu^*\mathcal Q$. Since $\mathcal E'$ is also generically globally generated, it is nef by (e), hence $\varphi(\mathcal E')$ is a nef locally free sheaf by (a). Since $\nu^*\mathcal Q/\varphi(\mathcal E')$ is torsion, it follows that $\nu^*\mathcal Q$ is nef by (d). Therefore $\mathcal Q$ is nef.
\end{proof}

The following result, see \cite[Theorem 0.1]{CP11} and \cite[Theorem 1.2]{CP15}, will be crucial in the proof of our main theorem. 

\begin{thm}\label{thm:CP11}
Let $X$ be a projective manifold, and let $(\Omega_X^1)^{\otimes m}\to\mathcal Q$ be a torsion free coherent quotient for some $m\geq1$. If $K_X$ is pseudoeffective, then $c_1(\mathcal Q)$ is pseudoeffective.
\end{thm}

The following lemma is well known.

\begin{lem} \label{van}
Let $Y$ be a projective manifold,  let $\mathcal M$ be a nef line bundle on $Y$, and let $\mathcal E$ be a nef and big locally free sheaf on $Y$. Then 
$$ H^q(Y, \omega_Y\otimes \mathcal E \otimes \det \mathcal E \otimes \mathcal M) = 0 \quad\text{for }q\geq1.$$
\end{lem} 

\begin{proof} 
Let $\pi\colon \bP(\mathcal E)\to Y$ be the projection morphism, and let $r$ be the rank of $\mathcal E$. Note that
$$\omega_{\bP(\mathcal E)}=\pi^*(\omega_Y\otimes\det \mathcal E)\otimes \sO_{\bP(\mathcal E)}({-}r)\quad\text{and}\quad\pi_*\sO_{\bP(\mathcal E)}(1)=\mathcal E,$$
hence for every $q\geq1$,
$$ H^q(Y,\omega_Y\otimes \mathcal E \otimes \det \mathcal E \otimes \mathcal M ) \simeq H^q\big(\bP(\mathcal E),\omega_{\bP(\mathcal E)}\otimes \sO_{\bP(\mathcal E)}(r+1) \otimes \pi^*\mathcal M\big)=0$$
by the Kawamata-Viehweg vanishing theorem. 
\end{proof} 

We often below use the following result \cite[Theorem 6.2]{Fuj83}, see also \cite[Theorem 1.4.40]{Laz04}.

\begin{thm}\label{thm:fujita}
Let $X$ be a projective variety of dimension $n$ and let $D$ be a nef divisor on $X$. Then for any coherent sheaf $\mathcal F$ on $X$ and for every $i$ we have
$$h^i\big(X,\mathcal F\otimes\OO_X(mD)\big)=O(m^{n-i}).$$
\end{thm}
 
 \subsection{The Cone conjecture}

Let $V$ be a real vector space equipped with a rational structure, let $\mcal C$ be a cone in $V$, and let $\Gamma$ be a subgroup of $\GL(V)$ which preserves $\mcal C$. A rational polyhedral cone $\Pi\subseteq\mcal C$ is a \emph{fundamental domain} for the action of $\Gamma$ 
on $\mcal C$ if $\mcal C=\bigcup_{g\in\Gamma}g\Pi$ and $\inte\Pi\cap \inte g\Pi=\emptyset$ if $g\neq\id$.

The Cone conjecture (for the nef cone) deals with the action of the automorphism group of $X$ on $\Nef(X)$, where $X$ is a $\Q$-factorial projective klt variety with $K_X\equiv0$. According to the original version by Morrison \cite{Mor93}, inspired by Mirror Symmetry, there exists a fundamental domain of the action of $\Aut(X)$ on $\Nef(X)^+$, which is the convex hull of the cone spanned by all \emph{rational} divisors in $\Nef(X)$. 

The version of the conjecture by Kawamata \cite{Kaw97} postulates that there exists a fundamental domain of the action of $\Aut(X)$ on $\Nef(X)^e$, which is the cone spanned by all \emph{effective} divisors in $\Nef(X)$. This version is more natural from the point of view of birational geometry, especially since it implies, in its most general form, that the number of minimal models of terminal varieties is finite \cite[Theorem 2.14]{CL14}.

On the other hand, we have the following consequence of \cite[Theorem 4.1, Application 4.14]{Loo14}, which is a result which belongs completely to the realm of convex geometry.

\begin{lem}
Let $X$ be a $\Q$-factorial projective klt variety with $K_X\equiv0$. Assume that there exists a polyhedral cone $\Pi\subseteq\Nef(X)^+$ such that $\Aut(X)\cdot\Pi$ contains the ample cone of $X$. Then $\Aut(X)\cdot\Pi=\Nef(X)^+$, and there exists a rational polyhedral fundamental domain for the action of $\Aut(X)$ on $\Nef(X)^+$.
\end{lem}

In particular, if we assume that Kawamata's version of the Cone conjecture holds, then necessarily Conjecture \ref{con:2} holds. In Theorem \ref{thm:nef+} we show that at least one part of Conjecture \ref{con:2} is true, that 
$$\Nef(X)^e\subseteq\Nef(X)^+.$$

We need the following result of Shokurov and Birkar, \cite[Proposition 3.2]{Bir11}; note that it is a careful application of the boundedness of extremal rays, which is a consequence of Mori's bend-and-break.

\begin{thm}\label{thm:ShokurovPolytope}
Let $X$ be a $\Q$-factorial projective variety, let $D_1,\dots,D_r$ be prime divisors on $X$ and denote $V=\bigoplus_{i=1}^r\R D_i\subseteq\Div_\R(X)$. Then the set
$$\mathcal N(V)=\{\Delta\in V\mid (X,\Delta)\text{ is log canonical and }K_X+\Delta\text{ is nef\,}\}$$
is a rational polytope.
\end{thm}

\begin{thm}\label{thm:nef+}
Let $X$ be a $\Q$-factorial projective klt variety with $K_X\equiv0$. Then 
$$\Nef(X)^e\subseteq\Nef(X)^+.$$
\end{thm}

\begin{proof}
Let $D$ be an $\R$-divisor whose class is in $\Nef(X)^e$, and let $V\subseteq\Div_\R(X)$ be the vector space spanned by all the components $D_1,\dots,D_r$ of $D$. Replacing $D$ by $\varepsilon D$ for $0<\varepsilon\ll1$, we may assume that $(X,D)$ is a klt pair, and in particular, with notation from Theorem \ref{thm:ShokurovPolytope}, $D\in\mathcal N(V)$. On the other hand, clearly $D\in\sum_{i=1}^r\R_+ D_i\subseteq V$. By Theorem \ref{thm:ShokurovPolytope}, the set 
$$\mathcal N(V)\cap\sum_{i=1}^r\R_+ D_i$$
is a rational polytope, hence $D$ is spanned by nef $\Q$-divisors.
\end{proof}

\section{Differentials with coefficients in a line bundle}

In this section, we prove several properties of nef line bundles on a  Calabi-Yau threefold. For future applications in Part II to this paper, we mostly do not restrict ourselves
to varieties with Picard number $2$ or line bundles with numerical dimension $2$. 

\begin{lem}\label{lem:hodge}
Let $X$ be a smooth projective surface and let $L$ and $M$ be divisors on $X$ such that $L^2=M^2=L\cdot M=0$. If $L$ and $M$ are not numerically trivial, then $L$ and $M$ are numerically proportional.
\end{lem} 

\begin{proof}
Let $H$ be an ample divisor on $X$. By the Hodge index theorem we have $\lambda=L\cdot H\neq0$ and $\mu=M\cdot H\neq 0$, and set $D=\lambda M-\mu L$. Then $D^2=D\cdot H=0$, hence $D\equiv 0$ again by the Hodge index theorem.
\end{proof}

\begin{lem}\label{lem:l-d}
Let $X$ be a smooth projective threefold with $H^1(X,\OO_X)=0$ and let $L$ be a nef divisor on $X$ with $\nu(X,L)=1$. Assume that $\kappa(X,L)=-\infty$ and let $D$ be a non-zero effective divisor on $X$. Then the divisor $L-D$ is not pseudoeffective.
\end{lem}

\begin{proof}
Denote $G=L-D$, and assume that $G$ is pseudoeffective. Denote $P=P_\sigma(G)$ and $N=N_\sigma(G)$, see\ \cite[Chapter III]{Nak04}. By \cite[Remark III.2.8 and paragraph after Corollary V.1.5]{Nak04} there is a set $Z\subseteq X$ which is a countable union of subvarieties of codimension at least $2$, such that $P|_C$ is nef for any curve $C \not \subseteq Z$. Therefore, if $H$ be a general very ample divisor on $X$, then $Z \cap H$ is a countable set, hence 
$P|_H$ is nef. In particular,
\begin{equation}\label{eq:restrictionNef}
(P|_H)^2\geq0.
\end{equation}

On the other hand, we have
$$0=(L|_H)^2=L|_H\cdot P|_H+L|_H\cdot N|_H+L|_H\cdot D|_H,$$
hence 
$$L|_H\cdot P|_H=L|_H\cdot N|_H=L|_H\cdot D|_H=0.$$
Now the Hodge index theorem implies $(P|_H)^2\leq0$, and hence $(P|_H)^2=0$ by \eqref{eq:restrictionNef}. Then Lemma \ref{lem:hodge} yields $P|_H\equiv\lambda L|_H$ for some real number $\lambda\geq0$, and hence $P\equiv\lambda L$ by the Lefschetz hyperplane section theorem. Note that $\lambda<1$ since $D$ is non-zero. Therefore, setting $E=\frac{1}{1-\lambda}(N+D)$, we obtain
$$L\equiv E,$$
and the Weil $\R$-divisor $E$ is effective. Let $E_1,\dots,E_r$ be components of $E$ and let $\pi\colon \Div_\R(X)\to N^1(X)_\R$ be the standard projection. Then $\pi^{-1}\big(\pi(L)\big)\cap\sum\R_+E_i$ is a rational affine polygon of $\sum\R E_i$ which contains $E$, hence there exists a rational point 
$$\textstyle E'\in \pi^{-1}\big(\pi(L)\big)\cap\sum\R_+E_i.$$ 
Therefore $L\equiv E'$, and consequently $L\sim_\Q E'$, which is a contradiction with $\kappa(X,L)=-\infty$.
\end{proof}

\begin{rem}  \label{remtriv} 
The assertion of Lemma \ref{lem:l-d} is obviously also true when $\nu(X,L) = 2$, provided that $\rho(X) = 2$. 
\end{rem} 

The following result has been claimed in \cite[3.1]{Wi94} when $\nu(X,L) = 2$. 

\begin{pro} \label{pro:2.3new} 
Let $X$ be a Calabi-Yau threefold and let $L$ be a nef divisor on $X$ such that $\kappa(X,L) = - \infty$. Then there is a positive integer $m_0$ such that
$$ H^0\big(X,\Omega^q_X \otimes \sO_X(mL) \big)=0\quad\text{for all } \vert m \vert \geq m_0\text{ and all }q.$$
\end{pro} 

\begin{proof} 
The result is proved in full generality within a broader context in \cite[Theorem 8.1]{LP18}; here we give a proof which works when $\nu(X,L)=1$ or when $\nu(X,L)=2$ and $\rho(X)=2$. The method was already used in \cite[Theorem 5.1]{HPR13}. Notice also that for $m \ll 0,$ the claim is clear: since $L$ is nef but not numerically trivial, we have ${-}L \cdot C < 0$ for all general curves cut out by hyperplane sections. 
 
Assume to the contrary that there exists $q$ such that 
$$ H^0\big(X,\Omega^q_X \otimes \sO_X(mL)\big) \neq 0$$
for infinitely many $m$. Every nontrivial section of $H^0\big(X,\Omega^q_X \otimes \sO_X(mL)\big)$ gives an inclusion $\sO_X(-mL) \to \Omega_X^q$, and consider the smallest subsheaf $\mathcal F \subseteq \Omega_X^q$ containing the images of all these inclusions. Let $r$ be the generic rank of $\mathcal F$. Then, without loss of generality, we may find infinitely many $r$-tuples $(m_1,\dots,m_r)$ such that the image of the map 
$$\sO_X({-}m_1L) \oplus\cdots\oplus \sO_X({-}m_rL) \to\mathcal F$$ 
has rank $r$. Taking determinants, we obtain infinitely many inclusions $\sO_X(-m'L) \to \det\mathcal F $. Let $F$ be a Cartier divisor such that $\OO_X(-F)$ is the saturation of $\det\mathcal F$ in $\bigwedge^r\Omega^q_X$. Therefore 
\begin{equation}\label{eq:infmany}
H^0\big(X,\OO_X(-F) \otimes \sO_X(m'L)\big) \ne 0\quad\text{for infinitely many }m'. 
\end{equation}
Consider the exact sequence
$$ 0 \to \OO_X(-F) \to \bigwedge^r\Omega^q_X \to \mathcal Q \to 0.$$
Since $\OO_X(-F)$ is saturated,  it follows that $\mathcal Q$ is torsion free, and hence $c_1(\mathcal Q)$ is pseudoeffective by Theorem \ref{thm:CP11}. As $\omega_X \simeq \sO_X$,  we deduce from the exact sequence above that $F=c_1(\mathcal Q)$, hence the divisor $F$ is pseudoeffective.

From \eqref{eq:infmany}, for every such $m'$ we obtain an effective divisor $N_{m'}$ such that 
$$N_{m'}+F\sim m'L.$$
Now Lemma \ref{lem:l-d} and Remark \ref{remtriv}  yield $N_{m'}=0$ for all $m'$,  hence some multiple of $L$ is linearly equivalent to $0$, a contradiction with $\kappa(X,L)=-\infty$.
\end{proof} 

\begin{cor}  \label{cor:growth order}
Let $X$ be a Calabi-Yau threefold, and let $L$ be a nef Cartier divisor with $\nu(X,L) = 2$ which is not semiample. Then there exists a positive integer $k$ such that 
$$ h^j\big(X,\Omega^q_X \otimes \sO_X(kmL)\big) = O(m)\quad\text{for all }q\text{ and }j.$$
\end{cor} 

\begin{proof} 
By Proposition \ref{prop1} and by Serre duality, for all $m$ we have 
\begin{align*}
\chi\big(X,\sO_X(mL)\big) &=0,\\ 
\chi\big(X,\Omega^1_X \otimes \sO_X(mL)\big) &= - \frac{1}{2}c_3(X),\\
\chi\big(X,\Omega^2_X \otimes \sO_X(mL)\big) &= \frac{1}{2}c_3(X).
\end{align*}
By Proposition \ref{pro:2.3new} and by Serre duality, 
$$ h^j\big(X,\Omega^q_X \otimes \sO_X(mL)\big) = 0 \quad\text{for }m \gg 0,\ j \in\{0,3\},\text{ and all }q.$$ 
Since 
$$ h^2\big(X,\Omega^q_X \otimes \sO_X(mL)\big) = O(m) $$
by Theorem \ref{thm:fujita}, we obtain $ h^1\big(X,\Omega^q_X \otimes \sO_X(mL)\big) = O(m) $.
\end{proof} 

The following lemma should be well-known. 

\begin{lem}\label{lem:33}
Let $X$ be a projective manifold and let $L$ be a pseudoeffective Cartier divisor on $X$. Let $h$ be a singular hermitian metric on $\OO_X(L)$
with semipositive curvature current  and  multiplier ideal sheaf $\mathcal I(h)$. Let $D$ be an effective Cartier divisor such that $\mathcal I(h)\subseteq\OO_X(-D)$. Then $L-D$ is pseudoeffective.
\end{lem}

\begin{proof} 
By \cite[Theorem 1.10]{DEL00} there exists an ample line bundle $G$ on $X$ such that $\OO_X(G+mL)\otimes \sI(h^{\otimes m}) $ is globally generated for all $m \geq 1$. Since 
$$\sI(h^{\otimes m})\subseteq \sI(h)^m\subseteq\OO_X(-mD),$$ 
where the first inclusion follows from \cite[Theorem 2.6]{DEL00}, for all $m \geq 1$ we have
$$ H^0\big(X,G+m(L-D)\big) \ne 0.$$
Hence $L-D=\lim\limits_{m\to\infty}\frac{1}{m}\big(m(L-D)+G\big)$ is pseudoeffective. 
\end{proof}

\begin{cor}  \label{cor:van}
Let $X$ be a Calabi-Yau threefold and let $L$ be a nef Cartier divisor on $X$ with $\kappa (X,L) = {-} \infty$. Then there exists a positive integer $m_0$ such that 
$$H^q(X,mL) = 0 \quad\text{for all }q\text{ and }m \geq m_0.$$ 
\end{cor} 

\begin{proof} 
If $\nu(X,L) = 2$, then $H^q(X,mL) = 0$ for $q\geq2$ and all $m\geq1$ by the Kawamata-Viehweg vanishing \cite[Corollary]{Kaw82}. Since $\chi(X,mL) = 0$ for all $m$ by Proposition \ref{prop1}(iii), and since $H^0(X,mL) = 0$ by assumption, we also have $H^1(X,mL)=0$ for $m\geq 1$. 

So we may assume that $\nu(X,L) = 1$. Let $h$ be a singular metric on $\OO_X(L)$ with semipositive curvature current. Let $\sI(h^{\otimes m})$ be the multiplier ideal of the associated metric $h^{\otimes m}$ on $\OO_X(mL)$ and denote by $V_m\subseteq X$ the subspace defined by $\sI(h^{\otimes m})$. 

The subspace $V_m$ cannot contain an effective divisor $D$: otherwise, by Lemma \ref{lem:33}, $mL-D$ would be pseudoeffective, which would contradict Lemma \ref{lem:l-d}. Thus $\dim V_m\leq1$. The Hard Lefschetz theorem \cite[Theorem 0.1]{DPS01} gives the surjection
$$H^0\big(X,\Omega^1_X \otimes \OO_X(mL) \otimes \sI(h^{\otimes m})\big)\to H^2\big(X,\OO_X(mL)\otimes \sI(h^{\otimes m})\big)$$
hence $H^2\big(X,\OO_X(mL)\otimes \sI(h^{\otimes m})\big)=0$ by Proposition \ref{pro:2.3new}.
From the long cohomology sequence associated to the exact sequence 
$$0\to \OO_X(mL)\otimes\sI(h^{\otimes m})\to \OO_X(mL)\to \OO_{V_m}(mL)\to0$$
we obtain
$$ H^2(X,mL) =0\quad\text{for }m\geq m_0.$$
Since $H^3(X,mL) = 0$ for $m\geq1$ by Serre duality, we conclude as above.
\end{proof}

In this context, we note the following:

\begin{thm} \label{thm:multiplier} 
Let $X$ be a Calabi-Yau threefold and let $L$ be a nef divisor on $X$ with $\nu(X,L)=1$. Assume that there is a singular metric $h$ on $\sO_X(L)$ with semipositive curvature current such that the multiplier ideal sheaf $\sI(h)$ is different from $\sO_X$. Then $L$ is semiample.
\end{thm}

\begin{proof} 
Let $V \subseteq X$ be the subspace defined by $\sI(h)$, and let $x$ be a closed point in $V$ with ideal sheaf $\sI_x$ in $X$. Let $\pi\colon \widehat X \to X$ be the blowup of $X$ at $x$ and let $E = \pi^{-1}(x) $ be the exceptional divisor. Let $\widehat h$ be the induced metric on $\pi^*\OO_X(L)$. By \cite[Proposition 14.3]{De01}, we have
$$\sI(\widehat h) \subseteq \pi^{-1}\sI(h)\cdot\OO_{\widehat X}\subseteq\pi^{-1}\mathcal I_x\cdot\OO_{\widehat X}=\OO_{\widehat X}(-E).$$ 
By Lemma \ref{lem:33}, the divisor $\pi^*L - E$ is pseudoeffective, hence $\pi^*L$ is semiample by Lemma \ref{lem:l-d}. 
\end{proof} 

\section{Log differentials}

The following result is crucial in the proof of Theorem \ref{MT2}.

\begin{lem}\label{lem:HN}
Let $X$ be a projective manifold and let $C \subseteq X$ be an irreducible curve on $X$ such that $K_X\cdot C\geq0$. Let $\mathcal L$ be an ample line bundle on $X$. Then there exists a positive integer $m_0$ such that for all $m \geq m_0$ and for a general element $D \in | \mathcal L^{\otimes m} |$, the sheaf $\Omega^1_X(\log D) |_C$ is nef. 
\end{lem}

\begin{proof} 
Let $\nu\colon \widetilde C \to C$ be the normalisation.  If $\Omega_X^1 |_C$ is nef, the assertion follows from the residue sequence and from Lemma \ref{lem:nefVB}(d). So we may assume that $\Omega^1_X |_C$ is not nef, or equivalently, that $\nu^*(\Omega^1_X |_C)$ is not nef. Since $\deg\nu^*(\Omega^1_X |_C)\geq0$ by assumption, the bundle $\nu^*(\Omega^1_X |_C)$ is not semistable by Lemma \ref{lem:nefVB}(c), and let $\mathcal F$ be its maximal destabilising subsheaf. Then $\mathcal F$ has positive slope, hence is ample by Lemma \ref{lem:nefVB}(c).

We obtain an exact sequence 
$$0 \to \mathcal F \to \nu^*(\Omega^1_X |_C) \to \mathcal G \to 0,$$
where $\mathcal G$ is a locally free sheaf on $\widetilde C$. Choose a positive integer $N$ and a very ample line bundle $\mathcal A$ on $\widetilde C$ with $\deg  \mathcal A = N$, such that $\mathcal G \otimes \mathcal A$ is globally generated. Then it follows that $\mathcal G \otimes \mathcal A'$ is nef for all ample divisors $\mathcal A'$ on $C$ of degree $\geq N$. 

Fix smooth points $x_1, \ldots x_N\in C$ and let $\mathcal I$ be the corresponding ideal sheaf; since $\nu$ is an isomorphism around $x_j$, we consider $x_j$ also as points on $\widetilde C$. We may choose finitely many (analytically) open sets $U_1, \ldots, U_m$ in $X$ such that:
\begin{enumerate}
\item[(a)] $U_i$ are pairwise disjoint,
\item[(b)] we fix trivialisations $\mathcal L|_{U_i}\simeq \OO_{U_i}$,
\item[(c)] $x_j \in \bigcup_{i=1}^m U_i$ for every $j=1,\dots, N$, 
\item[(d)] the complement of $\bigcup_{i=1}^m U_i$ in $X$ has measure $0$, 
\item[(e)] the complement of $C\cap\bigcup_{i=1}^m U_i$ in $C$ is finite and contains all the singular points of $C$. 
\end{enumerate}
Indeed, since $X$ is compact, there are finitely many (analytically) open subsets $V_1,\dots,V_m$ which cover $X$. We may assume that for each $j=1,\dots,N$ we have $x_j\notin\bigcup_{i=1}^m(\overline{V_i}\setminus V_i)$, where $\overline{V_i}$ is the closure of $V_i$ in the analytic topology. We may also assume that $C\cap \bigcup_{i=1}^m(\overline{V_i}\setminus V_i)$ is finite. Let $C^{\textrm{sing}}$ be the singular set of $C$. Then we set $U_1:=V_1\setminus C^{\textrm{sing}}$, and $U_{i+1}:=V_i\setminus\Big(\overline{\bigcup_{j=1}^{i-1} V_j}\cup C^{\textrm{sing}}\Big)$ for $i=1,\dots,m-1$.

Thus, by (a) and (b) we may speak of  the derivative  $ds(x)$ of a section $s$ of $\mathcal L$ at any point of $\bigcup_{i=1}^m U_i$. For $j=1,\dots,N$, fix
$$v_j \in \nu^*(\Omega^1_X |_C) \otimes \C(x_j) \setminus \mathcal F \otimes \C(x_j),$$
which we also view as elements of $\Omega^1_X |_C\otimes \C(x_j)$. Choose a positive integer $m_0 \geq N$ such that $\mathcal L^{\otimes m_0}$ is very ample  and  such that 
$$H^1(X,\mathcal I^2\otimes\mathcal L^{\otimes m})=0\quad\text{for all }m\geq m_0.$$ 
Fix $m \geq m_0$. Then this vanishing implies that the restriction map
$$H^0(X,\mathcal L^{\otimes m})\to H^0(X,\mathcal L^{\otimes m}\otimes\OO_X/\mathcal I^2)$$
is surjective, and hence there exists a section $s\in H^0(X,\mathcal L^{\otimes m})$ such that
\begin{equation}\label{eq:section}
s(x_j) = 0\quad\text{and}\quad ds(x_j) = v_j \quad\text{for every $j$.}
\end{equation}

\medskip

Now, let 
$$ M \subseteq H^0(X,\mathcal L^{\otimes m}) $$
be the subspace of all sections $s\in H^0(X,\mathcal L^{\otimes m})$, for which there exists points $y_1, \dots, y_N \in \widetilde C \cap \nu^{-1}\big(\bigcup_{i=1}^m U_i\big)$ such that for all $j$ we have
$$ s(y_j) = 0\quad\text{and}\quad  ds(y_j) \not \in \mathcal F \otimes \C(y_j).$$ 
Then $M \neq \emptyset $ by \eqref{eq:section} and by (c), and the set $M$ is clearly open. Therefore, by Bertini and by (e), there exists a smooth element $D\in|\mathcal L^{\otimes m}|$, meeting $C$ transversally at points $z_1,\dots,z_\ell$ (with $\ell \geq N$) in the locus $C\cap \bigcup_{i=1}^m U_i$, and such that $D$ is in $M$. By relabelling, we may assume that $A=\{z_1,\dots,z_N\} \subseteq C$ is the set of points such that 
\begin{equation}\label{eq:d}
d\varphi(z_i) \not \in \mathcal F \otimes \C(z_i),
\end{equation}
where $\varphi$ is the local equation of $D$ (in the given local trivialisation of $\mathcal L$).

\medskip

We claim that the sheaf $\sF$ is saturated in $\nu^*(\Omega^1_X(\log D) |_C)$ at the points of $A$. Granting the claim for the moment, let us see how it implies the lemma. We obtain the exact sequence
$$\textstyle 0 \to \mathcal F \to \nu^*(\Omega^1_X(\log D) |_C) \to \mathcal G' \to 0, $$
with 
$$\textstyle \mathcal G'/{\rm torsion} \simeq \mathcal G \otimes \OO_{\widetilde C}\big(\sum_{z_i\in A'} z_i\big),$$
where $A' \subseteq \widetilde C$ is a set containing $A$. By our choice of $N$, the vector bundle $\mathcal G \otimes \OO_{\widetilde C}\big(\sum_{z_i\in A'} z_i\big) $ is nef. Since $\sF$ is ample, the bundle $\nu^*(\Omega^1_X(\log D) |_C)$ is nef by Lemma \ref{lem:nefVB}(d).  

\medskip

Finally, we prove the claim. It is enough to show that for any $z_i \in A$, the linear map 
$$ \mathcal F \otimes \C(z_i)  \to \nu^*(\Omega^1_X(\log D) |_C) \otimes \C(z_i)$$
has rank $\rk\mathcal F$. Consider the diagram:
\[
\xymatrix{ \mathcal F \otimes \C(z_i) \ar[r] \ar[dr] & \nu^*(\Omega^1_X |_C) \otimes \C(z_i) \ar[d]^{\alpha} \\
 & \nu^*(\Omega^1_X(\log D) |_C) \otimes \C(z_i)
          }
\]
and note that $\ker\alpha=\C(z_i)d\varphi(z_i)$. Therefore, $\big(\mathcal F \otimes \mathbb C(z_i)\big)\cap\ker\alpha=0$ by \eqref{eq:d}, and the claim follows.
\end{proof} 

\section{A semiampleness criterion} \label{sec:ample} 

In this section we establish Theorem \ref{thm:introample}: 

\begin{thm} \label{ample} 
Let $X$ be a Calabi-Yau threefold with $c_3(X) \ne 0$ and let $L$ be a nef divisor on $X$ with $\nu(X,L) =2$. Let $G$ be a smooth ample divisor on $X$. Assume that there exists a very ample divisor $H$ and a positive integer $m_1$ such that for general $D \in | H |$ (so that $G+D$ has simple normal crossings) the following holds:
\begin{enumerate} 
\item[(i)]  the locally free sheaf $\Omega^1_X\big(\log(D+G)\big) \otimes \OO_X(mL)$ is nef for $m \geq m_1$, 
\item[(ii)] the divisor $L|_{D+G} $ is ample. 
\end{enumerate} 
Then $L$ is semiample. 
\end{thm} 

We note here that if for each curve $C$ in $X$ the sheaf $\Omega^1_X(\log D) \otimes \OO_X(mL)|_C$ is nef or the sheaf $\Omega^1_X(\log G) \otimes \OO_X(mL)|_C$ is nef, then (i) holds; see Step 3 of the proof of Proposition \ref{nef}.

\medskip

Most of this section is devoted to the proof of Theorem \ref{ample}. In order to prove Theorem \ref{ample},  we argue by contradiction and assume that $L$ is not semiample, hence that $\kappa(X,L) = - \infty$. 

A version of the following proposition was asserted in \cite[Theorem 2.3]{Wi94}, but we could not follow the first lines of the proof. 

\begin{pro} \label{pro:H2vanishing} 
Let $X$ be a Calabi-Yau threefold and let $L$ be a nef divisor on $X$ which is not semiample. Let $G$ be a smooth ample divisor and $X$. Assume that there exists a very ample divisor $H$ and a positive integer $m_1$ such that for general $D \in | H |$ (so that $G+D$ has simple normal crossings) the following holds:
\begin{enumerate} 
\item[(i)] the locally free sheaf $\Omega^1_X\big(\log(D+G)\big) \otimes \mathcal O_X(mL)$ is nef for $m \geq m_1$,  
\item[(ii)] the divisor $L|_{D+G}$ is ample. 
\end{enumerate} 
Then
$$H^2\big(X,\Omega^1_X \otimes  \OO_X(m L)\big) = 0\quad\text{for }m \gg 0.$$ 
\end{pro} 

\begin{proof} 
\emph{Step 1.}  
By our assumption, there exists a non-empty Zariski open set $B \subseteq | H |$ and a positive number $m_1$ such that 
\begin{equation}\label{eq:nef}
\E_{m}=\Omega^1_{X}\big(\log (D+G)\big) \otimes \OO_{X}(mL)\quad\text{is nef}
\end{equation}
for all $m \geq m_1$ and for all $D\in B$. Possibly further shrinking $B$,  we may furthermore assume that 
\begin{equation}\label{eq:20}
L|_{D+G}\quad\text{is ample for every }D\in B.
\end{equation}
We fix $D\in B$ and claim that
\begin{equation}\label{eq:15}
H^2\big(X, \Omega^1_X(\log(D+G)) \otimes \OO_X(\mu L)\big) = 0 \quad\text{for }\mu \gg 0.
\end{equation}
The claim immediately implies the proposition: indeed, tensoring the residue sequence \eqref{eq:residue1} with $\OO_X(\mu L)$ and taking cohomology gives the exact sequence
\begin{align*}
H^1\big(D, \OO_{D}(\mu L)\big) \oplus H^1(G,\OO_{G}(\mu L)\big) &\to  H^2\big(X, \Omega^1_X \otimes \OO_X(\mu L)\big) \\
&\to H^2\big(X, \Omega^1_X(\log(D+G)) \otimes \OO_X(\mu L)\big),
\end{align*}
hence it suffices to show
\begin{equation}\label{eq:155}
H^1\big(D, \OO_{D}(\mu L)\big) =   H^1\big(G, \OO_{G}(\mu L)\big) = 0.
\end{equation}
But $L|_D$ and $L|_G$ are ample by \eqref{eq:20}, hence \eqref{eq:155} follows by Serre vanishing as soon as $\mu$ is sufficiently large.

\medskip

\emph{Step 2.}
It remains to prove \eqref{eq:15}. Tensoring the standard exact sequence associated to $D+G$ with $\Omega^1_X\big(\log(D+G)\big) \otimes \OO_X(\mu L+D+G)$ and taking cohomology, we get the exact sequence
\begin{align*}
H^1\big(D+G&, \Omega^1_X(\log (D+G)) \otimes \sO_{D+G}(\mu L+D+G)\big)\\
& \to H^2\big(X,\Omega^1_X(\log (D+G)) \otimes \OO_X(\mu L)\big)\\
& \to H^2\big(X,\Omega^1_X(\log (D+G)) \otimes \sO_X(\mu L+D+G)\big).
\end{align*}
Hence, it suffices to show that for $\mu\gg0$ we have
\begin{equation}\label{eq:16}
H^1\big(D+G, \Omega^1_X(\log (D+G)) \otimes \sO_{D+G}(\mu L+D+G)\big) =0
\end{equation}
and
\begin{equation}\label{eq:17}
H^2\big(X,\Omega^1_X(\log (D+G)) \otimes \sO_X(\mu L+D+G)\big)=0.
\end{equation}
The equation \eqref{eq:16} follows from Serre vanishing. For \eqref{eq:17}, we may further assume that $m_1$ is so large, so that
\begin{equation}\label{eq:18}
10m^2L^2 \cdot (D+G) +5mL\cdot D\cdot G - (D+G) \cdot c_2(X) - c_3(X)>0.
\end{equation}
Denote $\mathcal E = \Omega^1_X\big(\log (D+G)\big) \otimes \OO_X(m_1L)$. By \cite[Chapter 3]{Ful98}, by Proposition \ref{prop1}(vi) and by \eqref{eq:18}, we have
\begin{align*}
c_1\big(\sO_{\bP(\mathcal E)}(1)\big)^5 = s_3(\mathcal E)> 0,
\end{align*}
and therefore, by \eqref{eq:nef} the line bundle $\sO_{\bP(\mathcal E)}(1)$ is nef and big on $\bP(\mathcal E)$. Noticing that $\det \mathcal E=\OO_X(D+G+3m_1L)$, we have
$$ \Omega^1_X\big(\log (D+G)\big) \otimes \sO_X(\mu L+D+G)  =  \mathcal E\otimes \det \mathcal E \otimes \OO_X\big((\mu-4m_1)L\big),$$
and \eqref{eq:17} follows by Lemma \ref{van}. This finishes the proof.
\end{proof} 

\begin{pro}  \label{pro:H1vanishing} 
Let $X$ be a Calabi-Yau threefold and let $L$ be a nef divisor on $X$. Assume that there exists a smooth very ample divisor $D$ such that the divisor $L|_D $ is ample. Then
$$ H^1\big(X,\Omega^1_X \otimes \OO_X(-m L)\big) = H^2\big(X,\Omega^2_X \otimes \OO_X(m L)\big) = 0\quad\text{for }m \gg 0.$$
\end{pro} 

\begin{proof} 
Tensoring the residue sequence \eqref{eq:residue3} associated to $\Omega_X^1(\log D)$ with $\sO_X(-m L)$, and taking cohomology, we obtain the exact sequence 
\begin{align}
H^1\big(X,  \Omega^1_X(\log D) \otimes \sO_X(-D-m L)\big) &\to H^1\big(X, \Omega^1_X \otimes \sO_X(-m L)\big)\label{eq:align}\\
&\to H^1\big(D,\Omega_D^1 \otimes \sO_D(-m  L)\big).\notag
\end{align}
Now, since $D+m L$ is ample, we have 
$$H^1\big(X,  \Omega^1_X(\log D) \otimes \sO_X(-D-m L)\big)=0$$
by \cite[Corollary 6.4]{EV92}, and as $L|_D$ is ample by assumption, by Serre duality and Serre vanishing we have
$$H^1\big(D,\Omega_D^1 \otimes \sO_D(-m  L)\big)\simeq H^1\big(D,\Omega_D^1 \otimes \sO_D(m  L)\big)=0 \quad\text{for }m \gg 0.$$
Therefore, \eqref{eq:align} gives $H^1\big(X,\Omega^1_X \otimes \OO_X(-m L)\big) = 0$, which together with Serre duality proves our assertion. 
\end{proof} 

\begin{proof}[Proof of Theorem \ref{ample}]
Assume that $L$ is not semiample, and in particular, that $\kappa(X,L)=-\infty$. By Proposition \ref{pro:2.3new} and by Serre duality, there exists a positive integer $m_0$ such that for all integers $m$ with $|m| \geq m_0$ we have 
$$h^0\big(X,\Omega^1_X \otimes \OO_X(mL)\big)=h^3\big(X,\Omega^1_X \otimes \OO_X(mL)\big)=0.$$ 
Therefore, by Propositions \ref{prop1}(iv) and  \ref{pro:H2vanishing}, for $m \gg 0$ we have
$$- \frac{1}{2}c_3(X)=\chi\big(X,\Omega^1_X \otimes \OO_X(mL)\big) = -h^1\big(X,\Omega^1_X \otimes \OO_X(mL)\big)\leq0, $$
and by Propositions \ref{prop1}(iv) and \ref{pro:H1vanishing}, for $m\gg0$ we have 
$$- \frac{1}{2}c_3(X)=\chi\big(X,\Omega^1_X \otimes \OO_X(-mL)\big) = h^2\big(X,\Omega^1_X \otimes \OO_X(-mL)\big)\geq0.$$
In total, we obtain $c_3(X) = 0,$ contradicting our assumption. This finishes the proof. 
\end{proof}  

\section{Proof of the Main Theorem} \label{sec:nu=2}

In this section we prove Theorem \ref{MT2}: 

\begin{thm}  \label{thm:nu=2} 
Let $X$ be a Calabi-Yau threefold with $\rho(X) = 2$ and let $L$ be a nef Cartier divisor on $X$.  Suppose that $\nu(X,L)=2$ and that $c_3(X) \ne 0$. Then $L$ is semiample.
\end{thm} 

By Theorem \ref{ample} and Lemma \ref{generic}(b), it suffices to verify the condition (i) in Theorem \ref{ample}.
 
\begin{nt}\label{notation:1}
Let $D$ be a smooth very ample divisor on $X$. 

Denote by $B$ a non-empty Zariski open affine subset of the linear system $\PS\big(H^0(X,\OO_X(D))\big)$ contained in the locus of smooth elements of that linear system. Set $\mathcal{X}=X\times B$, let $\pi\colon \mathcal X\to B$ be the projection map and let $\mathcal L$ be the pullback of $L$ by the projection $\mathcal X\to X$. Let $\mathcal D\subseteq\mathcal X$ be the universal family of divisors parametrised by $B$. Note that $\mathcal D$ is a smooth divisor in $\mathcal X$. We consider the relative logarithmic cotangent sheaf 
$$\Omega^1_{\mathcal X/B}(\log\mathcal D)$$ 
with log poles along $\mathcal D$, which is a locally free sheaf of rank $3$ on $\mathcal X$. Denote by 
$$T_{\mathcal X/B}(-\log\mathcal D)$$ 
its dual. For every point $b\in B$, denote $X_b=\pi^{-1}(b)$, $D_b=\mathcal D\cap X_b$ and $L_b=\mathcal L|_{X_b}$. Note that $X_b = X$, $L_b = L$ and  
$$\Omega^1_{\mathcal X/B}(\log\mathcal D)|_{X_b}=\Omega^1_{X_b}(\log D_b)\quad\text{and}\quad T_{\mathcal X/B}({-}\log\mathcal D)|_{X_b}=T_{X_b}({-}\log D_b).$$ 
For each positive integer $m$, denote
$$\mathcal E_m=\Omega^1_{\mathcal X/B}(\log\mathcal D)\otimes\OO_\mathcal X(m\mathcal L)\quad\text{and}\quad \mathcal E_{m,b}=\mathcal E_m|_{X_b}.$$
In the remainder of the section, we freely shrink $B$ if necessary. 
\end{nt}

\begin{lem} \label{RR} 
Let $X$ be a Calabi-Yau threefold with $\rho(X) = 2$ and let $L$ be a nef Cartier divisor on $X$ with $\nu(X,L)=2$ which is not semiample. Assuming Notation \ref{notation:1}, the following holds:
\begin{enumerate}
\item[(a)] there exist a positive integer $m_0$ and a positive constant $C$, such that for every $b\in B$  and for all $m\geq m_0$, we have 
$$ h^0(X_b,\mathcal E_{m,b}) \geq Cm^2,$$
\item[(b)] for a general point $b\in B$ we have 
$$ h^0(X_b,\mathcal E_{m,b})=\frac{1}{2}m^2L^2\cdot D+O(m)\quad\text{and}\quad h^1(X_b,\mathcal E_{m,b})=O(m).$$
\end{enumerate}
\end{lem} 

\begin{proof}
For (a), Proposition \ref{prop1}(v) gives 
$$ \chi(X_b,\mathcal E_{m,b})  = \frac{1}{2}m^2L^2 \cdot D + O(m), $$
where $O(m)$ does not depend on $b$. Fix $b_0\in B$, and for each positive integer $m$, let 
$$U_m=\big\{b\in B\mid h^2(X_b,\mathcal E_{m,b})\leq h^2(X_{b_0},\mathcal E_{m,b_0})\big\}.$$
These sets are Zariski open in $B$ by upper-semicontinuity; denote 
$$U=\bigcap\limits_{m\geq1}U_m.$$
Since
$$ h^2(X_{b_0,}\mathcal E_{m,b_0} )=O(m)$$
by Theorem \ref{thm:fujita}, there exists a constant $C_1>0$ such that
$$h^2(X_b,\mathcal E_{m,b})\leq C_1m\quad\text{for all $m$ and for all }b\in U.$$
Therefore, there is a constant $C>0$ and a positive integer $m_0$ such that for all $m\geq m_0$ and all $b\in U$ we have
$$h^0(X_b,\mathcal E_{m,b})\geq \frac{1}{2}m^2L^2 \cdot D + O(m)-C_1m\geq Cm^2.$$
We conclude by upper-semicontinuity, since $U$ is dense in $B$.

For (b), we have $h^0\big(X_b,\Omega^1_{X_b} \otimes \OO_{X_b}(mL_b)\big) = 0 $ by Proposition \ref{pro:2.3new} and $ h^1\big(X_b,\Omega^1_{X_b} \otimes \OO_{X_b}(mL_b)\big) =O(m)$ by Corollary \ref{cor:growth order}. Tensoring the residue sequence \eqref{eq:residue2} associated to $D_b$ by $\OO_{X_b}(mL_b)$ and taking the long cohomology sequence, since $L_b|_{D_b}$ is ample by Lemma \ref{generic}(b), Serre vanishing gives the exact sequence
\begin{align*}
0&\to H^0(X_b,\sE_{m,b})\to H^0\big(D_b,\OO_{D_b}(mL_b)\big)\\
&\to H^1\big(X_b,\Omega^1_{X_b} \otimes \OO_{X_b}(mL_b)\big)\to H^1(X_b,\sE_{m,b})\to0.
\end{align*}
This immediately implies $h^1(X_b,\mathcal E_{m,b})=O(m)$ and
$$ h^0(X_b,\sE_{m,b}) = h^0\big(D_b, \OO_{D_b}(mL)\big) + O(m). $$ 
Riemann-Roch and Serre vanishing give 
$$ h^0\big(D_b,\OO_{D_b}(mL_b)\big) = \chi\big(D_b,\OO_{D_b}(mL_b)\big) = \frac{1}{2} m^2 (L_b|_{D_b})^2 + O(m), $$
and (b) follows from the last two equations.
\end{proof}

\begin{pro} \label{pro:min}  
Let $X$ be a Calabi-Yau threefold with $\rho(X) = 2$ and let $L$ be a nef Cartier divisor on $X$ with $\nu(X,L)=2$  which is not semiample.  Let $\mathcal A$ be the set of all ample prime divisors on $X$, and let $A\in\mathcal A$ be an element such that $A \cdot c_2(X)$ is minimal. If $M$ is an integral divisor such that $M \sim_{\mathbb Q} aA+bL$ with $a > 0$, then $a \geq 1$. 
\end{pro} 

\begin{proof} 
Assume $0 < a < 1$. Setting 
$$M_0 := M-\lfloor b\rfloor L,$$
it suffices to prove the claim for $M_0$. Hence, replacing $M$ by $M_0$, we may assume from the beginning that $0 \leq b < 1$. Then $M$ is ample as $A$ is ample and $L$ is nef.  By \eqref{eq:RR2}, we have
$$\chi(X,M)=\frac{1}{6}M^3+\frac{1}{12}M\cdot c_2(X).$$
Since $M$ is ample, by Proposition \ref{pro:Miyaoka} we have $\chi(X,M)>0$, and hence $h^0(X,M) > 0$ by Kodaira vanishing. Pick $E \in |M|$. Since $L\cdot c_2(X)=0$ by Proposition \ref{prop1}(ii), we have
$$ E \cdot c_2(X) = (aA+bL) \cdot c_2(X) < A \cdot c_2(X),$$
hence $E\notin\mathcal A$ by the choice of $A$. Write $E=\sum a_i E_i$, where $a_i$ are positive integers and $E_i$ are prime divisors. By Lemma \ref{lem1} we have $E_i \cdot c_2(X) > 0$ for all $i$, hence 
$$E_i \cdot c_2(X) \leq E \cdot c_2(X) < A \cdot c_2(X). $$ 
However, since $\rho(X) = 2$ and since $L$ lies on the boundary of the pseudoeffective cone, at least one $E_{i_0}$ must be ample. This contradicts the choice of $A$. 
\end{proof} 

\begin{lem}\label{lem:boundsoncoefficients}
Let $X$ be a Calabi-Yau threefold with $\rho(X) = 2$ and let $L$ be a nef Cartier divisor on $X$ with $\nu(X,L)=2$ which is not semiample. Assume that
$$ H^0\big(X,\Omega^q_X \otimes \OO_X(kL)\big)=0\quad\text{for all }q\text{ and all }k.$$
Assume Notation \ref{notation:1}, and assume additionally that 
$$D\sim_\Q \alpha A+\beta L,$$ 
where $A$ is as in Proposition \ref{pro:min} and $\alpha,\beta\in\Q_{>0}$. Let $S\sim aD+bL$ be an effective divisor on $X$ such that 
$$\OO_X(S) \subseteq \Omega^r_X(\log D) \otimes \OO_X(mL)$$ 
for some $r$ and $m$. Then:
\begin{enumerate}
\item[(i)] $\alpha\geq1$ and $0\leq a\leq 1-\frac1\alpha$, and $a>0$ if $b\neq0$,
\item[(ii)] the divisor $(1-a)D+(m-b)L$ is pseudoeffective.
\end{enumerate} 
\end{lem}

\begin{proof}
Notice that $a\geq0$, since $S$ is effective and $L$ lies on the boundary of the pseudoeffective cone. If $b\neq0$, then additionally $a\neq0$ since $L$ is not effective. The inclusion 
$$\OO_X(S) \subseteq \Omega^r_X(\log D) \otimes \OO_X(mL)$$ 
implies that 
$$ H^0\big(X, \Omega^r_X(\log D)(-D) \otimes \sO_X((1-a)D+(m-b)L )\big) \ne 0, $$
so that, via the inclusion $\Omega^r_X(\log D)(-D) \subseteq \Omega^r_X$,
\begin{equation}\label{eq:323}
H^0\big(X,\Omega^r_X \otimes \sO_X((1-a)D+(m-b)L)\big) \ne 0.
\end{equation}
Hence we obtain an inclusion
$$\sO_X\big((a-1)D+(b-m)L\big)\to \Omega^r_X. $$
Since $X$ is a Calabi-Yau threefold, the divisor
$$(1-a)D+(m-b)L\sim_\Q (1- a) \alpha A + \big(m - b - (a - 1) \beta\big) L$$
is pseudoeffective by Theorem \ref{thm:CP11}. Thus,
\begin{equation}\label{eq:157}
(1 - a)\alpha \geq 0,
\end{equation}
and hence $a \leq  1$. If $a = 1$, then \eqref{eq:323} implies
$$ H^0\big(X,\Omega^r_X \otimes \sO_X((m - b)L)\big) \neq 0,$$
which contradicts our assumption. Then from \eqref{eq:157} and from Proposition \ref{pro:min} we have $\alpha\geq1$ and $(1 - a) \alpha \geq 1$, and the lemma follows.
\end{proof}

\begin{pro} \label{prop:globgen} 
Let $X$ be a Calabi-Yau threefold with $\rho(X) = 2$ and let $L$ be a nef Cartier divisor on $X$ with $\nu(X,L)=2$  which is not semiample.  Assume Notation \ref{notation:1}, and assume additionally that 
$$D\sim_\Q \alpha A+\beta L,$$ 
where $A$ is as in Proposition \ref{pro:min} and $\alpha,\beta\in\Q_{>0}$. 

Then, possibly after shrinking $B$, there exists a positive number $n_1$ and an algebraic set $\mathcal V \subsetneq\mathcal X$ such that for all irreducible curves $ C \not  \subseteq \mathcal V$ which are contracted by $\pi$, the restricted bundle $\sE_{n_1} |_C$ is nef. 
\end{pro} 
 
\begin{proof}
We note first that by Proposition \ref{pro:2.3new}, by passing to a multiple of $L$ we may assume that
\begin{equation}\label{eq:158}
H^0\big(X,\Omega^q_X \otimes \OO_X(kL)\big)=0\quad\text{for all }q\text{ and all }k.
\end{equation}
We prove the proposition in several steps.

\medskip

\emph{Step 1.}
For each $m$, let $U_m\subseteq B$ be the locus of points where the sheaf $\pi_*\mathcal E_m$ is locally free and has the base change property, 
and denote 
$$U=\bigcap\limits_{m\geq1}U_m.$$ 
Fix $b_0\in U$. By Lemma \ref{RR}(a), there exist a positive constant $C$ and a positive integer $n_1$ such that 
\begin{equation} \label{eq:asymp}   
h^0(X_{b_0},\mathcal E_{m,b_0}) \geq C m^2  \geq 2\quad\text{for }m \geq n_1.
\end{equation} 
Since $B$ is affine, by the definition of $U$ the map
$$ H^0(\mathcal X,\mathcal E_m)\to H^0(X_{b_0},\mathcal E_{m,b_0})$$
is surjective for all $m$, cf.\ \cite[III.12]{Har77}. In particular, 
 $$\rk {\rm Im}(\pi^*\pi_*(\sE_m) \to \sE_m) \geq1\quad\text{for }m\gg0.$$
Let $\sS_m$ be the saturation of ${\rm Im}(\pi^*\pi_*(\sE_m) \to \sE_m)$ in $\sE_m$, and let $ \sS_{m,b}$ be the saturation, in $\sE_{m,b}$, of the sheaf generated by the global sections of $\sE_{m,b}$. Then 
$$\sS_{m,b} = (\sS_m |_{X_b})^{**},$$
and in particular, for every $b\in B$ we have $\sS_m |_{X_b} = \sS_{m,b}$ generically.  Let 
$$r_m = \rk \sS_m = \rk \sS_{m,b}.$$ 
By possibly replacing $b_0$, we may assume that
\begin{equation}\label{eq:165}
\sS_m |_{X_{b_0}} = \sS_{m,b_0} \quad\text{for all }m,
\end{equation} 
see for instance \cite[Proposition 5.2]{GKP16}. The sheaf $\det \sS_{m,b_0}$ is a line bundle by \cite[Proposition 1.9]{Ha80}. Since $\sS_{m,b_0}$ is generically generated by its definition and by \eqref{eq:asymp}, so is its determinant. In other words, $\det\mathcal S_{m,b_0}$ is effective, say $\det \sS_{m,b_0}  \simeq \mathcal O_{X_{m,b_0}}(M)$ with an effective divisor $M.$ 
Write
\begin{equation}\label{eq:161}
M \sim_{\mathbb Q} a_m D_{b_0} + b_mL_{b_0}
\end{equation}
with rational numbers $a_m$ and $b_m$. Since 
$$\det \sS_{m,b_0}\subseteq \Omega^{r_m}_{X_{b_0}}(\log D_{b_0}) \otimes \OO_{X_{b_0}}(r_mmL_{b_0})$$
and \eqref{eq:158} holds, Lemma \ref{lem:boundsoncoefficients} gives $\alpha\geq1$, and also that
\begin{equation}\label{eq:59}
0\leq a_m \leq 1 - \frac{1}{\alpha} \quad\text{for all } m,
\end{equation}
and that
\begin{equation}\label{eq:162}
\text{the divisor }(1-a_m)D_{b_0}+(r_mm-b_m)L_{b_0}\text{ is pseudoeffective.}
\end{equation}

\medskip

\emph{Step 2.} 
Suppose $r_m=3$ for some $m$, and we set 
$$\mathcal V=\Supp\Coker(\pi^*\pi_*\mathcal E_m\to\mathcal E_m).$$
Then for every curve $C\not\subseteq\mathcal V$ contracted by $\pi$, the restriction $\mathcal E_m|_C$ is generically generated, hence nef by Lemma \ref{lem:nefVB}(e).

\medskip

\emph{Step 3.}
Hence we may assume that $r_m\leq2$ for all $m$. In this step we assume that $r_m = 1$ for infinitely many $m$, and in particular, $\det \sS_{m,b_0}=\sS_{m,b_0}$. 
 
By \eqref{eq:RR2} and \eqref{eq:161}, together with Proposition \ref{prop1}(ii) we have
\begin{equation}\label{eq:56}
\chi(X_{b_0},\sS_{m,b_0}) = \frac{1}{2} a_mb_m^2 L^2 \cdot D + \frac{1}{2} a_m^2 b_m L \cdot D^2 + \frac{1}{6} a_m^3 D^3+\frac{1}{12}a_m D\cdot c_2(X).
\end{equation}
If $b_m<0$ for infinitely many $m$, then these $b_m$ are bounded from below, since $a_m$ are bounded from above and the divisors $a_m D_{b_0} + b_mL$ are pseudoeffective. 
But then all $\sS_{m,b_0}$ lie in a compact subset of $\Pic(X)$, hence take only finitely many values, which contradicts the fact that $h^0(X_{b_0},\sS_{m,b_0})$ grows quadratically with $m$ by \eqref{eq:asymp}.

Therefore, 
\begin{equation}\label{eq:163}
b_m\geq0\quad\text{ for }m\gg0,
\end{equation}
and a similar argument shows that, by \eqref{eq:162}, there exists a positive constant $c$ such that 
\begin{equation}\label{eq:164}
b_m-m<c\quad\text{for all }m.
\end{equation}
From \eqref{eq:164}, \eqref{eq:asymp} and by Lemma \ref{lem:boundsoncoefficients} we have $a_m>0$, which implies that $\mathcal S_{m,b_0}$ is ample for $m\gg0$. Hence, by \eqref{eq:56}, by Kodaira vanishing, and by Lemma \ref{RR}(b) we have
\begin{multline*}
h^0(X_{b_0},\sS_{m,b_0}) = \frac{1}{2} m^2 L^2 \cdot D + O(m)\\
=\frac{1}{2} a_mb_m^2 L^2 \cdot D + \frac{1}{2} a_m^2 b_m L \cdot D^2 + \frac{1}{6} a_m^3 D^3+\frac{1}{12}a_m D\cdot c_2(X),
\end{multline*}
which is a contradiction by \eqref{eq:59}, \eqref{eq:163} and \eqref{eq:164}. 

\medskip
 
\emph{Step 4.} 
We are now reduced to the case $r_m = 2$ for all large $m$.  

Let $p_X\colon \mathcal X \to X$ be the projection to $X$ and recall that $\sL = p_X^*L$. Consider the subsheaves 
$$ \sS_{m+1} \subseteq  \sE_{m+1}\quad\text{and}\quad \sS_m \otimes \OO_\mathcal X(\mathcal L) \subseteq \sE_m \otimes \OO_\mathcal X(\mathcal L)=\mathcal E_{m+1}.$$
Introduce the torsion free rank $1$ sheaves  $\mathcal Q_m = \mathcal E_m /  \sS_m$ and consider the induced maps
$$ \alpha_m\colon \sS_{m+1} \to \mathcal E_m\otimes\OO_\mathcal X(\mathcal L)\to \mathcal Q_m \otimes \OO_\mathcal X(\mathcal L) $$ 
and
$$ \beta_m\colon \sS_m \otimes \OO_\mathcal X(\mathcal L)\to\mathcal E_{m+1} \to\mathcal Q_{m+1}.$$ 

Suppose first that $\beta_m$ is not the zero map for some $m$, and in particular, $\beta_m$ is generically surjective since $\mathcal Q_{m+1}$ is of rank $1$. We define
$$\mathcal W=\Supp\Coker(\pi^*\pi_*\mathcal E_m\to\mathcal E_m)\cup\Supp\Coker(\pi^*\pi_*\mathcal E_{m+1}\to\mathcal E_{m+1}).$$
Then for every curve $C\not\subseteq \mathcal W$, the sheaf  $\sS_m|_C$ is generically globally generated. Since $\beta_m$ induces the generically surjective morphism 
$$\mathcal S_m|_C\to \big(\mathcal Q_{m+1}\otimes\OO_{\mathcal X}(-\mathcal L)\big)|_C,$$ 
the line bundle $\big(\mathcal Q_{m+1}\otimes\OO_{\mathcal X}(-\mathcal L)\big)|_C$ is nef by Lemma \ref{lem:nefVB}(f), hence $\mathcal Q_{m+1}|_C$ is nef. If $\nu\colon\widetilde C\to C$ is the normalisation, we have the exact sequence
$$ \nu^*(\sS_{m+1}|_C) \xrightarrow{\gamma_{m+1}} \nu^*(\sE_{m+1}|_C) \to \nu^*(\mathcal Q_{m+1}|_C) \to 0,$$ 
and since $\gamma_{m+1}\big(\nu^*(\sS_{m+1}|_C)\big)$ is nef by Lemma \ref{lem:nefVB}(a,e), the line bundle $\nu^*(\sE_{m+1}|_C)$ is nef by Lemma \ref{lem:nefVB}(d). Therefore, $\sE_{m+1}$ is $\pi$-nef outside of $\mathcal W$, confirming the claim of the proposition. 
%Since $\sS_m$ is generically generated, $\mathcal Q_{m+1} \otimes \OO_{\mathcal X}(\mathcal L)$ is generically generated. Since $\mathcal L$ is $\pi-$nef, 
%$\mathcal Q_{m+1} $ is $\pi$-nef outside of a finite union $U$ of subvarieties, hence Lemma \ref{lem:nefVB}(d) and the exact sequence
%\begin{equation} \label{lari} 0 \to \sS_{m+1}  \to \sE_{m+1} \to \mathcal Q_{m+1} \to 0 \end{equation}
%imply that $\sE_{m+1}$ is $\pi$-nef outside of an algebraic set, confirming the claim of the proposition. In fact, consider a curve $C \not \subset  U$ contracted by $\pi$
%and restrict Sequence \ref{lari} to $C$ to obtain an exact sequence
%$$ (\sS_{m+1})_{| C} { \buildrel{a_m} \over {\longrightarrow}} (\sE_{m+1})_{| C}  \to (\mathcal Q_{m+1})_{| C} \to 0.$$ 
%Since the image of $a_m$ is generically generated, we may apply  Lemma \ref{lem:nefVB}(d,e) to conclude that $(\sE_{m+1})_{| C}$ is nef. 

We note here for later use, that if $\mathcal E_m$ is not $\pi$-nef outside of an algebraic set for all $m$, then
\begin{equation}\label{eq:159}
h^0(\mathcal X,\mathcal Q_m)=0\quad\text{for all }m.
\end{equation}
Otherwise there would exist a nontrivial morphism $\OO_\mathcal X\to\mathcal Q_m$, which would imply as above that $\mathcal Q_m$ and $\mathcal E_m$ are $\pi$-nef outside of an algebraic set.  

Thus we may assume that $\beta_m$ is the zero map for all large $m$, and thus 
$$\sS_m \otimes \OO_\mathcal X(\mathcal L)\subseteq \mathcal S_{m+1}\quad\text{for}\quad m\gg0.$$ 
Therefore $\sS_m \otimes \OO_\mathcal X(\mathcal L)$ and $\mathcal S_{m+1}$ are generically equal, thus $\alpha_m$ is generically the zero map, which implies that $\alpha_m$ is the zero map since the image of $\alpha_m$ is a torsion free coherent subsheaf of $\mathcal Q_m \otimes \OO_\mathcal X(\mathcal L)$. We conclude that there exists $m_1\gg0$ such that
$$ \sS_{m+1} = \sS_m \otimes \OO_\mathcal X(\mathcal L) \quad\text{for}\quad m \geq m_1,$$
and thus 
$$ \sS_m = \sS_{m_1} \otimes \OO_\mathcal X\big((m-m_1) \sL\big).$$
Restricting this last equation to $X_{b_0}$, by \eqref{eq:165} we have 
$$ \sS_{m,b_0}  = \sS_{m_1,b_0} \otimes \sO_{X_{b_0}}\big((m-m_1) L_{b_0}\big).$$
We claim that 
\begin{equation}\label{eq:160}
h^0\big(X_{b_0}, \sS_{m_1,b_0} \otimes \sO_{X_{b_0}}(mL_{b_0})\big) = \frac{1}{2} m^2 L^2 \cdot c_1(\sS_{m_1,b_0}) +O(m).
\end{equation}
This immediately implies the proposition: indeed, by \eqref{eq:59} we have that $c_1(\sS_{m_1,b_0}) = a_{m_1} D_{b_0} + b_{m_1} L_{b_0}$, where $a_{m_1}$ is bounded away from $1$, hence
$$h^0\big(X_{b_0}, \sS_{m_1,b_0} \otimes \sO_{X_{b_0}}(mL_{b_0})\big) = \frac{1}{2}a_{m_1} m^2 L^2 \cdot D +O(m).$$
Since $h^0(X_{b_0},\sE_{m_1+m,b_0}) =  h^0\big(X, \sS_{m_1,b_0} \otimes \OO_{X_{b_0}}(mL_{b_0})\big)$, this is a contradiction by Lemma \ref{RR}(b).

It remains to prove \eqref{eq:160}. First note that the Chern classes $c_1\big(\sS_{m_1,b_0} \otimes \sO_X(mL_{b_0})\big) $ and $c_2\big(\sS_{m_1,b_0} \otimes \sO_X(mL_{b_0})\big) $ are computed as in the locally free case, since $\sS_{m_1,b_0} $ is locally free outside a finite set by \cite[Corollary 1.4]{Ha80}, whereas 
$$c_3\big(\sS_{m_1,b_0}\otimes \sO_X(mL_{b_0})\big) = c_3(\sS_{m_1,b_0}), $$ 
cf.\ \cite[p.\ 130]{Ha80}. Thus Riemann-Roch for coherent sheaves \cite{OTT81} gives 
\begin{equation}\label{eq:167}
\chi\big(X_{b_0}, \sS_{m_1,b_0} \otimes \sO_{X_{b_0}}(mL_{b_0})\big) = \frac{1}{2} m^2 L^2 \cdot c_1(\sS_{m_1,b_0}) +O(m).
\end{equation}
Define sheaves $\mathcal Q_{m,b_0}$ by the short exact sequences
\begin{equation}\label{eq:166}
0\to \sS_{m,b_0}\to \sE_{m,b_0}\to \mathcal Q_{m,b_0}\to 0,
\end{equation}
and observe that $\mathcal Q_{m,b_0}=\mathcal Q_m|_{X_{b_0}}$ by \eqref{eq:165}. Since $B$ is affine, \eqref{eq:159} shows that the rank of the sheaf $\pi_*\mathcal Q_m$ is zero for every $m$, hence by possibly changing $b_0$, we have $h^0(X_{b_0},\mathcal Q_{m,b_0}) = 0$ for every $m$. Then \eqref{eq:166} and Lemma \ref{RR}(b) give 
$$h^1(X_{b_0},\mathcal S_{m,b_0})\leq h^1(X_{b_0},\mathcal E_{m,b_0})=O(m),$$
and since $h^2(X_{b_0},\sS_{m,b_0})=h^3(X_{b_0},\sS_{m,b_0}) = O(m)$ by Theorem \ref{thm:fujita}, the claim follows from \eqref{eq:167}.
\end{proof} 

\begin{pro}\label{nef}
Let $X$ be a Calabi-Yau threefold with $\rho(X) = 2$ and let $L$ be a nef Cartier divisor on $X$ with $\nu(X,L)=2$  which is not semiample.  Assume Notation \ref{notation:1}, and assume additionally that 
$$D\sim_\Q \alpha A+\beta L,$$ 
where $A$ is as in Proposition \ref{pro:min} and $\alpha,\beta\in\Q_{>0}$. Then there exists a smooth ample divisor $G$ on $X$ such that, after shrinking $B$, such that the divisor $D_b+G$ has simple normal crossings for every $b\in B$ and there exists a positive number $m_2$ such that 
$$\Omega^1_{X_b}\big(\log (D_b+G)\big) \otimes \OO_{X_b}(mL_b)\quad\text{is nef}$$
for all $m \geq m_2$ and for all $b\in B$. 
\end{pro}

\begin{proof}
\noindent \emph{Step 1.} 
In this step, we show that after possibly shrinking $B$, there exist a positive integer $n_2$ and an algebraic set $\mathcal C\subseteq\mathcal X$ of codimension $\geq2$, such that $\mathcal E_m|_C$ is nef for every $m\geq n_2$ and for every curve $C\not\subseteq\mathcal C$ which is contracted by $\pi$.

Indeed, by  Proposition \ref{prop:globgen} we find a positive integer $n_1$ and an algebraic set $\mathcal V \subseteq \mathcal X$  such that $\mathcal E_{n_1} |_C $ is nef for all curves $C \not  \subseteq \mathcal V$ contracted by $\pi$.  Let $\mathcal V_1,\dots,\mathcal V_k$ and $\mathcal W_1,\dots,\mathcal W_l$ be the codimension $1$, respectively codimension $2$ irreducible components of $\mathcal V$. Set $\pi_j=\pi|_{\mathcal V_j}$. By possibly shrinking $B$, we may assume that $\mathcal V$ does not contain any fibre of $\pi$, and that each $\pi_j$ is flat. 

For each $j$, the line bundle $\mathcal L|_{\mathcal V_j}$ is clearly $\pi_j$-nef. Moreover, it is also $\pi_j$-big: indeed, for each $b\in B$, the set $\mathcal V_j\cap X_b$ is a surface in $X_b$, and since $\nu(X,L)=2$, we have $(\mathcal L|_{\mathcal V_j\cap X_b})^2 >0$ by Lemma \ref{lem:surface}. Therefore, by Kodaira's trick, there exist a $\pi_j$-ample $\Q$-divisor $\mathcal A_j$ and an effective $\Q$-divisor $\mathcal B_j$ such that 
$$\mathcal L|_{\mathcal V_j}\sim_\Q \mathcal A_j+\mathcal B_j.$$ 
For $k$ sufficiently divisible, the sheaf $\mathcal E_{n_1}|_{\mathcal V_j}\otimes \OO_{\mathcal V_j}(k\mathcal A_j)$ is $\pi_j$-globally generated, and we conclude that for every curve $C$ contracted by $\pi$ which is not contained in the locus 
$$\mathcal C=\bigcup\Supp\mathcal B_j\cup\bigcup\mathcal W_i,$$ 
the sheaf $\mathcal E_{n_1+k}|_C$ is nef. We set $n_2=n_1+k$.

\medskip  
  
\emph{Step 2.}  
In the second step we show that there exist a positive integer $n_3$ and finitely many curves $C_1, \dots, C_s$ on $X$, such that $\sE_m |_C$ is nef for all $m\geq n_3$ and for all curves $C \notin \{C_1, \dots, C_s\}$. 
%The proposition then follows immediately: indeed, by Lemma \ref{lem:HN} by possibly shrinking $B$ we may assume that there exists a positive integer $n_4$ such that the sheaves 
%$\Omega^1_{X_b}(\log D_b)|_{C_i}$ are nef for all $m \geq n_4$, for all $b\in B$ and for all $i\in\{1,\dots,s\}$. Then we set $m_2:=\max\{n_3,n_4\}$.
%In fact,  fixing $m \geq m_2,$ the set of all $b \in B$ such that $\mathcal E_{m,b}$ is nef, is open in $B$, \cite{Mwk92}. Hence, possibly shrinking $B,$ it suffices to find $m$ such that $\mathcal E_{m,b_0}$ is nef for one $b_0.$ This number $m$ exists by Lemma \ref{lem:HN}. 

In order to prove the claim, let $\mathcal C_1,\dots,\mathcal C_s$ be the irreducible components of $\mathcal C$. Fix $j$, and consider the normalisation $\nu_j\colon\mathcal C_j^\nu\to \mathcal C_j$ and the Stein factorisation $\alpha_j\colon\mathcal C_j^\nu\to B_j$ of $\pi|_{\mathcal C_j}\circ\nu_j$. 
\[
\xymatrix{ 
\mathcal C_j^\nu \ar[dr]_{\alpha_j} \ar[r]^{\nu_j} & \mathcal C_j \ar[r]^{\pi|_{\mathcal C_j}} & B\\
& B_j \ar[ur] &  
}
\]
After possibly shrinking $B$, we may assume that $\alpha_j$ is a flat morphism. Let $B_{jk}$ be the connected components of $B_j$ and let $\mathcal C_{jk}^\nu=\alpha_j^{-1}(B_{jk})$ and $\mathcal C_{jk}=\nu_j(\mathcal C_{jk}^\nu)$. Then $\mathcal C_{jk}^\nu$ is irreducible since $\alpha_j$ has connected fibres and therefore $\mathcal C_{jk}$ is an irreducible component of $\mathcal C_j$, which maps onto $B$. Now, for fixed $k$ we have that
$$\nu_j^*\mathcal L\cdot \alpha_j^{-1}(b)=\mathcal L\cdot(\nu_j)_*\big(\alpha_j^{-1}(b)\big)\quad\text{ is constant for }b\in B_{jk}.$$ 
So if this constant is positive, then $\mathcal L|_{\mathcal C_{jk}}$ is $\pi|_{\mathcal C_{jk}}$-ample. In particular, increasing $m_2$, we conclude that $\mathcal E_m|_{\mathcal C_{jk}}$ is $\pi|_{\mathcal C_{jk}}$-nef. 

Therefore, if $\mathcal E_m|_C$ is not nef, then $C$ belongs to a family $\mathcal C_{jk}$ on which $L$ is numerically trivial. Since there are only countably many $L$-trivial curves on $X$ by Lemma \ref{generic}(a), each family $\mathcal C_{jk}$ must be constant. This shows the claim of Step 2.

\medskip

\emph{Step 3.} 
Now by Lemma \ref{lem:HN} we may choose a smooth ample divisor $G$ on $X$ such that $\Omega^1_{X_b}(\log G) |_{C_j}$ is nef for all $j=1,\dots,s$. Let $C$ be a curve on $X$. If $C \neq C_j$ for all $j$, then tensoring the residue sequence \eqref{eq:residue2} associated to $D_b$ and $D_b+G$ by $\OO_{X_b}(mL_b)$ for $m \geq n_3$, we obtain the exact sequence
\begin{multline*}
\Omega^1_{X_b}(\log D_b) \otimes \OO_{X_b}(mL_b)|_C\\
\to \Omega^1_{X_b}(\log (D_b+G)) \otimes \OO_{X_b}(mL_b)|_C \to \OO_G(mL_b)|_C \to 0,
\end{multline*}
which yields the nefness of $\Omega^1_{X_b}(\log (D_b+G)) \otimes \OO_{X_b}(mL_b)|_C$ by Lemma \ref{lem:nefVB}(d). On the other hand, if $C = C_j$ for some $j$, then tensoring the residue sequence \eqref{eq:residue2} associated to $G$ and $D_b+G$ by $\OO_{X_b}(mL_b)$ for $m \geq n_3$, we obtain the exact sequence
\begin{multline*}
\Omega^1_{X_b}(\log G)|_C \otimes \OO_{X_b}(mL_b) \\
\to \Omega^1_{X_b}(\log (D_b+G)) \otimes \OO_{X_b}(mL_b)|_C \to \OO_{D_b}(mL_b)|_C \to 0,
\end{multline*}
which yields the nefness of $\Omega^1_{X_b}(\log (D_b+G)) \otimes \OO_{X_b}(mL_b)|_C$ again by Lemma \ref{lem:nefVB}(d). This finishes the proof.
\end{proof}

Combining the results of Sections \ref{sec:ample} and \ref{sec:nu=2} together with Lemma \ref{generic}(b) we immediately obtain the following:

\begin{cor} \label{cor:vanishings}
Let $X$ be a Calabi-Yau threefold with $\rho(X) = 2$ and let $L$ be a nef line bundle on $X$ with $\nu(X,L)=2$  which is not semiample. Then
\begin{enumerate}
\item[(i)] $c_3(X) = 0$,
\item[(ii)] there exists a positive integer $k$ such that for all integers $m$ we have
$$ H^q\big(X,\Omega^p_X \otimes \sO_X(mkL)\big) = 0\quad\text{for all }p,q.$$
\end{enumerate} 
\end{cor} 

\begin{rem}
Let $X$ be a Calabi-Yau threefold with $\rho(X) = 2$ and let $L$ be a nef line bundle on $X$ with $\nu(X,L)=2$ which is not semiample. Let $D$ be any smooth divisor. 
Then from Corollary \ref{cor:vanishings} and the standard residue sequences associated to $\Omega^1_X(\log D)$ and $\Omega^2_X(\log D)$, we deduce that there is a positive integer $m_0$ such that for every integer $m \geq m_0$ the residue maps 
$$ H^0\big(X, \Omega^1_X(\log D) \otimes \OO_X(mL)\big) \to  H^0\big(D,\OO_D(mL)\big) $$
and 
$$ H^0\big(X, \Omega^2_X(\log D) \otimes \OO_X(mL)\big) \to  H^0\big(D,\Omega^1_D \otimes \OO_D(mL)\big) $$
are isomorphisms. We expect that on a simply connected manifold this can never happen.
\end{rem} 

\section{Proof of Corollary \ref{cor:rational}}\label{section4}

In this section we prove Corollary \ref{cor:rational}. We follow the arguments in \cite{Og93}. Choose a rational boundary ray $R$ of $\Nef(X)$ and a Cartier divisor $L$ such that $R=\R_+L$. If $L^2 \ne 0$, then $L$ is semiample by Theorem \ref{MT2}, and hence produces an elliptic fibration $f\colon X \to S$. Then $X$ contains a rational curve by \cite{Pe91,Og93}.  If $L^2 = 0$, then we claim that the second ray $R'$ is also rational. Indeed, fix any divisor $L'$ such that $R'=\R_+L'$. If $(L')^3 > 0$, then the claim follows from \cite{Wi89} or from the Cone theorem. If $(L')^3 = 0$, fix an ample divisor $H$ and consider the cubic polynomial 
$$ p(t) = (L+tH)^3\in \Z[t]. $$ 
There is a unique real number $t_0 < 0$ such that $-(L+t_0H) \in R'$. Thus $p(t_0) = 0$. Since $p$ has integer coefficients and a double zero at $0$, the number $t_0 $ is rational and hence the ray $R'$ is rational. Since $t_0$ is a simple zero of $p$, we have $(L')^2 \ne 0$, and we conclude as above.

\section{A semiampless criterion, II}\label{sec:semicriterion2}

In this section we generalise Theorem \ref{ample} to the case where there do exist surfaces $S$ such that $L \cdot S = 0 $. In this case we have $\rho (X) \geq 3$ by Lemma \ref{generic}. Since Theorem \ref{ample} is easier and since the relevant applications are restricted to the case of Picard number $\rho(X) = 2$ so far, we decided for the benefit of the reader to treat the more general case separately. 

\begin{rem}
Let $X$ be a Calabi-Yau threefold and let $L$ be a nef divisor on $X$ with $\nu(X,L)=2$. Then there are {\it only finitely many} surfaces $S$ with $L\cdot S=0$. Indeed, it suffices to show that all such surfaces are linearly independent in the finite-dimensional vector space $N^1(X)_\R$. Let $E_1,\dots E_r$ be a finite collection of surfaces on $X$ with $L\cdot E_i=0$ for all $i$, and assume that $\sum e_i E_i\equiv 0$ for some real numbers $e_i$ which are not all zero. We may assume that $e_i\geq0$ for $i\leq k$, and $e_i<0$ otherwise. Denote 
$$A :=\sum_{i\leq k}e_iE_i\quad\text{and}\quad B:= - \sum_{i>k}e_iE_i,$$ 
and note that $A\neq0$ and $B\neq0$.
Let $H$ be a hyperplane section such that $B\cdot H$ is an effective $1$-cycle $C$ not contained in $\Supp A$. Since $\nu(X,L) =2$, we have $(A|_H)^2<0$ by the Hodge index theorem on $H$, and hence
$$A\cdot C=A\cdot B\cdot H=A^2\cdot H<0,$$
a contradiction.
\end{rem}

\subsection{Reduction to good Calabi-Yau models}

\begin{dfn}\label{dfn:goodCY}
Let $X$ be a Calabi-Yau threefold and let $L$ be a nef divisor on $X$ with $\nu(X,L) = 2$. A prime divisor $S$ on $X$ is \emph{orthogonal to $L$} if $L\cdot S=0$. Let $S_1, \ldots, S_r$ be all the prime divisors on $X$ orthogonal to $L$. If there exist a birational morphism $\phi\colon X \to Z$ and a Cartier divisor $M$ on $Z$ such that:
\begin{enumerate}
\item[(a)] $Z$ is a $\Q$-factorial projective threefold with canonical singularities and $\omega_Z \simeq \OO_Z$, 
\item[(b)] $\phi$ contracts precisely the divisors $S_j$,
\item[(c)] $L\sim\phi^*M$,
\end{enumerate}  
then we call $\phi$ a \emph{good Calabi-Yau model for $L$}.
\end{dfn}

The starting observation is that every Calabi-Yau threefold $X$ together with a nef divisor $L$ can be modified to a good Calabi-Yau model:

\begin{thm}\label{flop}
Let $X$ be a Calabi-Yau threefold and let $L$ be a nef divisor on $X$ with $\nu(X,L) =2$. Let $S_1, \ldots, S_r$ be all the prime divisors on $X$ orthogonal to $L$. Then there is a diagram
\[
\xymatrix{ X \ar@{-->}[r]^{\pi}  & \widetilde X \ar[d]^{\phi} \\
 & X',
}
\]
where $\pi$ is an isomorphism in codimension $1$, $\pi_*L$ is a nef  divisor, and $\phi$ is a good Calabi-Yau model for $\pi_*L$.
\end{thm}

\begin{proof} 
This is \cite[Proposition 1.1]{Wi94}, and here we include a different proof. 

Denote $S=\sum_{j=1}^r S_j$. We first claim that $N_\sigma(S)=S$, where we use the notation from \cite[Chapter III]{Nak04}. Indeed, assume otherwise. Then $P_\sigma(S)\neq0$, and if $H$ is a general very ample irreducible surface on $X$, then $P_\sigma(S)|_H$ is nef. On the other hand, we have $L\cdot P_\sigma(S)=0$, and the Hodge index theorem implies $P_\sigma(S)^2\cdot H<0$, a contradiction.

Let $\varepsilon$ be a small positive rational number such that the pair $(X,\varepsilon S)$ is terminal. Consider a Minimal Model Program (MMP) for $(X,\varepsilon S)$. Since $N_\sigma(S)=S$, this MMP contracts the whole $S$ by \cite[Th\'eor\`eme 1.2]{Dru11}. We claim that this MMP is $L$-trivial, and in particular, $L$ stays nef on every step of this MMP. Indeed, as the analysis below shows, it suffices to prove this on the first step of the process, say $\varphi\colon X\dashrightarrow Y$. Assume $\varphi$ is a divisorial contraction and let $C$ be a curve contracted by $\varphi$; the case of a flip is analogous. Since $C\cdot S<0$, we must have $C\subseteq\Supp S$, and in particular, $L\cdot C=0$. Therefore, there exists a divisor $L_Y$ on $Y$ such that $L\sim\varphi^*L_Y$, hence $L_Y\cdot \varphi_*S_j=0$ for every $j$. 

Let $X'$ be the result of this MMP. An easy analysis of discrepancies shows that $X'$ is canonical, and the only geometric valuations with discrepancy zero correspond to $S_1,\dots,S_r$. Let $\phi\colon \widetilde X\to X'$ be a $\Q$-factorial terminalisation of $X'$. Then the induced birational map $\pi\colon X\dashrightarrow \widetilde X$ can be written as a composition of flops, and $\widetilde X$ is smooth by \cite[Theorem 2.4]{Kol89}. This finishes  the proof of the theorem.
\end{proof} 

\begin{rem} 
Given a Calabi-Yau threefold $X$ and a nef line bundle $\mathcal L$ on $X$, the condition $\kappa (X,\mathcal L) \geq 0$ is clearly invariant under isomorphisms in codimension one. Therefore, with notation from Theorem \ref{flop}, we may always pass from $X$ to $\widetilde X$ and therefore assume that $\phi$ is a morphism already on $X$.
\end{rem} 

In analogy with locally free sheaves, we say that a torsion free sheaf $\mathcal E$ is nef if the line bundle $\OO_{\PS(\mathcal E)}(1)$ is nef; we refer to \cite{Anc82} for details.  Moreover, $\sE$ is big if for some resolution $\pi\colon \widetilde X \to X$ such that $\widetilde {\mathcal E} := \pi^*\mathcal E / {\rm torsion} $ is locally free, the line bundle $\OO_{\PS(\widetilde {\mathcal E})}(1) $ is big. It is easy to see that this notion does not depend on the choice of $\pi$. 

\medskip

The goal of this section is the following result.

\begin{thm}\label{topology}
Let $X$ be a Calabi-Yau threefold and let $L$ be a nef divisor on $X$ with $\nu(X,L)=2$, such that there exists a birational morphism $\phi\colon X \to Z$ which is a good Calabi-Yau model for $L$. Let $S_1, \ldots, S_r$ be all the prime divisors on $X$ orthogonal to $L$. Let $g_j$ be the irregularity of a resolution of $S_j$. If
\begin{enumerate}
\item[(i)] there exists a general very ample divisor $G$ on $Z$, a very ample divisor $H$ on $Z$ and a positive integer $m_1$ such that for general $D \in | H |$ (so that $ D+G$ has simple normal crossings on the smooth locus of $Z$), the reflexive sheaf 
$$\Omega^{[1]}_Z\big(\log (D+G)\big) \otimes \OO_Z(m\phi_*L)$$
is nef for $m \geq m_1$, 
\item[(ii)] the divisor $(\phi_*L)|_{D + G}$ is ample, and
\item[(iii)] $\frac{c_3(X)}{2} \neq r-\sum_{j=1}^r g_j$,
\end{enumerate}
then $L$ is semiample.
\end{thm}

\begin{rem}
We discuss some special cases of Theorem \ref{topology}. 
\begin{enumerate}
\item[(a)] When $r=0$, we recover Theorem \ref{thm:introample}.
\item[(b)] If all surfaces $S_j$ are rational, then the condition (2) becomes $c_3(X) \neq 2r$.
\end{enumerate}
\end{rem} 

\subsection{Auxiliary results}

\begin{lem} \label{lem:nef} 
Let $E$ be a smooth projective surface and let $\pi\colon E \to B$ a surjective morphism to a curve $B$ with a general fibre isomorphic to $\PS^1$. Let $D \ne 0$ be a simple normal crossings divisor on $E$ and let $F$ be a general fibre of $\pi$ which meets $D$ transversally.
\begin{enumerate}
\item[(a)] If $D \cap F \ne \emptyset, $ then $\Omega^1_E(\log D) |_F \otimes \OO_F(1)$ is nef. 
\item[(b)] If $D \cap F$ contains at least two points, then $\Omega^1_E(\log D) |_F$ is nef. 
\end{enumerate} 
\end{lem} 
 
\begin{proof} 
We have the exact sequence 
$$ 0 \to \Omega^1_E |_F \to \Omega^1_E(\log D)|_F \to Q \to 0 $$
where $Q$ is a skyscraper sheaf supported on $F \cap D$. Composing with the canonical morphism $N^*_{F/E} \to \Omega^1_E |_F$, we obtain the exact sequence
$$ 0 \to N^*_{F/E} \to \Omega^1_E(\log D)|_F \to Q' \to 0, $$
where $Q'$ contains $\Omega^1_F$ as proper subsheaf. Write $\OO_F(a):=Q'/{\rm torsion}$. 

For (a), by Lemma \ref{lem:nefVB}(d) it suffices to show that $\OO_F(a)\otimes\OO_F(1)$ is nef. If $a \leq -2$, then we would obtain a (generically) surjective morphism $\Omega^1_E(\log D) |_F \to \Omega^1_F$. Dually, $T_F \subseteq T_E(- \log D) |_F$, and hence the canonical morphism $T_F \to N_{D/E} |_F$ vanishes, which contradicts the assumption that $D$ and $F$ meet transversally.

For (b), it suffices to show that $Q'$ is nef. Fix two distinct points $p,q\in D \cap F$. If $a = {-}1$, then by the same argument as above, we obtain an injective morphism
$$ T_F \otimes \OO_F(-p) \to T_E({-}\log D) |_F, $$ 
so that the induced morphism $T_F \otimes \OO_F(-p) \to N_{D/E}|_F $ vanishes at $q$, contradicting the assumption that $D$ and $F$ meet transversally at $q$. 
\end{proof} 

In the same way we prove the following:

\begin{lem} \label{lem:nef2}
Let $D \ne 0$ be a simple normal crossings divisor on $\PS^2$ and let $\ell \subseteq\PS^2$ be a line meeting $D$ transversally. Then $\Omega^1_{\PS^2}(\log D)|_\ell \otimes \OO_\ell(1)$ is nef. 
\end{lem} 

The following lemma is certainly well-known; we include a short proof, lacking a suitable reference.

\begin{lem} \label{lem:surface1} 
Let $S$ be a  normal quasi-projective surface with only rational double points as singularities. Let $\pi\colon \widehat S \to S$ be a resolution with the exceptional divisor $F = \sum_{j=1}^N F_j$. Then $R^1\pi_*\Omega^1_{\widehat S}(\log F)= 0$.
\end{lem}

\begin{proof} 
The problem being local, we may assume that $S$ is a Stein space with one singularity $s$. Moreover, we may assume that $F$ is a deformation retract of $\widehat S$. Pushing forward via $\pi$ the residue sequence
$$ 0 \to  \Omega^1_{\widehat S} \to \Omega^1_{\widehat S}(\log F) \to \bigoplus\nolimits_{j=1}^N \OO_{F_j} \to 0$$ 
and using the canonical isomorphism $\pi_*\Omega^1_{\widehat S} \simeq \pi_*\Omega^1_{\widehat S}(\log F) $, see  \cite[Theorems 1.4 and 16.1]{GKKP11}, we obtain the exact sequence 
$$ 0 \to \bigoplus_{j=1}^N \pi_*\OO_{F_j} \to R^1\pi_*\Omega^1_{\widehat S} \to R^1\pi_*\Omega^1_{\widehat S}(\log F) \to \bigoplus_{j=1}^N R^1\pi_*\OO_{F_j} = 0,$$
where the last vanishing follows from the rationality of the singularity $s$. Since $ \bigoplus_{j=1}^N \pi_*\OO_{F_j}$ is a  skyscraper sheaf of length $N$, it suffices to show that the skyscraper sheaf  $R^1\pi_*\Omega^1_{\widehat S}$ has length at most $N$. Since 
$$H^1(\widehat S, \Omega^1_{\widehat S}) \simeq H^0(S,R^1\pi_*\Omega^1_{\widehat S}),$$
it suffices to show that $h^1(\widehat S, \Omega^1_{\widehat S}) \leq N$.

Consider the Fr\"olicher spectral sequence
$$ E_1^{p,q} = H^q(\widehat S, \Omega^p_{\widehat S})\Longrightarrow H^{p+q}(\widehat S,\C). $$
Note that $E_1^{0,1} = H^1(\widehat S, \OO_{\widehat S}) = 0$ since $s$  is a rational singularity and $S$ is Stein, and moreover, 
$$ E_1^{2,1} = H^1\big(\widehat S, \OO_{\widehat S}(K_{\widehat S})\big) = R^1\pi_*\OO_{\widehat S}(K_{\widehat S}) =0$$
by the Grauert-Riemenschneider vanishing. Therefore, $E_\infty^{1,1}\simeq H^1(\widehat S, \Omega^1_{\widehat S})$. On the other hand, since $F$ is a deformation retract of $\widehat  S$, we have 
$$ h^2(\widehat S, \mathbb C) = h^2(F, \mathbb C) = N,$$
and the lemma follows.
\end{proof} 

\begin{lem} \label{lem:vanish1}
Let $Z$ be a $\Q$-factorial projective threefold with canonical singularities, and let $\pi\colon Y \to Z $ be a resolution such that the exceptional set $ E = \sum_{j=1}^s  E_j$ is a simple normal crossings divisor on $Y$. Then:
\begin{enumerate}
\item[(a)] $R^2\pi_*\Omega^1_Y(\log E) = 0$,
\item[(b)] the support of the sheaf $ R^1\pi_*\Omega^1_Y(\log E)$ is zero-dimensional.
\end{enumerate}
\end{lem} 

\begin{proof}
For (a), the vanishing \cite[Theorem 14.1]{GKKP11} reduces the claim to 
$$ R^2\pi_*\big(\Omega^1_Y(\log E)|_E\big) = 0,$$ 
hence we need to show  
$$ R^2\pi_*\big(\Omega^1_Y(\log E)|_{E_j}\big) = 0$$ 
for all components $E_j$ of $E$. Clearly, this needs to be proved only for those $E_j$ with $\dim \pi(E_j) = 0$, in which case we need to show 
$$H^2\big(E_j, \Omega^1_Y(\log E)|_{E_j}\big)=0.$$
Now clearly $H^2(E_j,\mathcal O_{E_j}) = 0.$ Using the canonical exact sequence
$$ 0 \to \Omega^1_{E_j} (\log D) \to \Omega^1_Y(\log E) |_{E_j} \to \OO_{E_j} \to 0 $$
with $D := E - E_j$,  it  therefore suffices to show that 
$$H^2\big(E_j,  \Omega^1_{E_j}(\log D)\big) = 0,$$
or dually, that
\begin{equation} \label{eq:last}   
H^0\big(E_j, T_{E_j}(-\log D) \otimes \omega_{E_j}\big) = 0.
\end{equation} 
By \cite[Corollary 1.5]{HM07a}, $E_k$ are uniruled surfaces. If $E \not\simeq \bP^2$, let $\ell \subseteq E_j$ be a general ruling line, so that $\ell\simeq\PS^1$ with $\ell^2 =0$. If $E \simeq \bP^2$, let $\ell$ be a general line. Then by Lemmas \ref{lem:nef} and \ref{lem:nef2}, the sheaf
$$  \Omega^1_{E_j}(\log D) |_\ell \otimes \OO_\ell(1) $$
is nef. This gives the vanishing \eqref{eq:last}.

\medskip

For (b), we may assume that the singular locus of $Z$ has only components of dimension $1$. Let $y$ be a general point on such a component. By \cite[Corollary 1.14]{Rei80},  locally analytically around $y$ we have $Z \simeq S \times \Delta$, where $S$ is a surface with only rational double points, and $\Delta$ is a small disk. Possibly shrinking $\Delta$, we may therefore assume that $Y = \widehat S \times \Delta$, where $\tau\colon\widehat S \to S$ is a resolution and $\pi = \tau \times \id$, and we denote by $\widehat p_1\colon Y \to \widehat  S$ and by $p_1\colon Z \to S$ the projections. Let $F = \sum F_j$ be the exceptional divisor of $\tau$, so that $E = F \times \Delta$. Since 
$$ \Omega^1_Y(\log E) \simeq \widehat p_1^*\Omega^1_{\widehat S}(\log F) \oplus \OO_Y,$$
and since $Z$ has rational singularities, it suffices to show that 
$$ R^1\pi_*\widehat p_1^*\Omega^1_{\widehat S}(\log F) = 0.$$
Since by the flat base change we have 
$$ R^1\pi_*\widehat p_1^*\Omega^1_{\widehat S}(\log F) \simeq p_1^*R^1\tau_*\Omega^1_{\widehat S}(\log F),$$
the desired vanishing follows from Lemma \ref{lem:surface1}. 
\end{proof}

\subsection{Cohomology of log differentials}

Our first aim is generalisations of Lemma \ref{van} and Proposition \ref{pro:H2vanishing}.

\begin{lem} \label{lem:prepvan}
Let $X$ be a normal projective Gorenstein threefold and let $\mathcal E$ be a reflexive sheaf of rank $r$ on $X$ such that $\det \mathcal E$ is locally free. Assume that $\mathcal E $ is big and nef. Then
$$ H^2(X, \mathcal E \otimes \det \mathcal E \otimes \mathcal M \otimes \omega_X) =  0$$
for any nef line bundle $\mathcal M.$ 
\end{lem} 

\begin{proof} 
Choose a resolution of singularities $ \pi\colon \widetilde X \to X$ such that the sheaf $ \widetilde {\mathcal  E}  := \pi^*\mathcal E / {\rm torsion}$ is locally free. Since $\widetilde {\mathcal E}$ is big and nef, we have
$$H^q\big(\widetilde X, \widetilde {\mathcal E} \otimes \det \widetilde {\mathcal E} \otimes \pi^*\mathcal M \otimes \omega_{\widetilde X}\big) = 0 $$
for $q \geq 1$ and for any nef line bundle $\mathcal M$ on $X$ by Lemma \ref{van}. This implies
$$ R^q\pi_*\big(\widetilde {\mathcal E} \otimes \det \widetilde {\mathcal E} \otimes  \omega_{\widetilde X}\big) = 0 $$
for $q \geq 1$ by \cite[Lemma 4.3.10]{Laz04}, hence
$$ H^q\big(X, \pi_*\big(\widetilde {\mathcal E} \otimes \det \widetilde {\mathcal E} \otimes \omega_{\widetilde X}\big) \otimes \mathcal M\big)  = 0$$
for $q \geq 1$ by the Leray spectral sequence. Now, 
$$ \pi_*\big(\widetilde {\mathcal E} \otimes \det \widetilde {\mathcal E} \otimes  \omega_{\widetilde X}\big) \subseteq \mathcal E \otimes \det \mathcal E \otimes \omega_X, $$
and the cokernel  $\mathcal Q$ of this inclusion is supported on a set of dimension at most $1$. Hence the result follows from the vanishing $H^q(X,\mathcal Q \otimes \mathcal M) = 0$ for $q\geq2$, and from the long exact sequence in cohomology.
\end{proof} 

\begin{thm} \label{amp2}
Let $X$ be a $\Q$-factorial threefold with canonical singularities and with $\omega_X  \simeq \OO_X$. Let $L$ be a nef divisor on $X$ with $\nu(X,L)=2$. Let $G$ be a general very ample divisor on $X$. Assume that there exists a very ample divisor $H$ and a positive integer $m_1$ such that for general $D \in | H |$ (so that $D+G$ has simple normal crossings on the smooth locus of $X$) the following holds:
\begin{enumerate} 
\item[(i)] the reflexive sheaf $\Omega^{[1]}_X\big(\log(D+G)\big) \otimes \mathcal O_X(mL)$ is nef for $m \geq m_1$,  
\item[(ii)] the divisor $L|_{D+G}$ is ample. 
\end{enumerate} 
If $L$ is not semiample, then 
$$H^2\big(X,\Omega^{[1]}_X \otimes \OO_X(mL)\big) = 0 \quad \textrm{for }m \gg 0.$$
\end{thm} 

\begin{proof} 
The proof is an easy adaptation of the proof of Proposition \ref{pro:H2vanishing}, using Lemma \ref{lem:prepvan} instead of Lemma \ref{van}. Only two issues need a word of explanation. 

First, since $D$ is general in its linear system, the complement of the set $(D+G)^{\circ} = X_{\rm reg} \cap (D\cup G)$ is finite, and the residue sequence reads
$$ 0 \to \Omega^1_{X_{\rm reg}}  \to \Omega^1_{X_{\rm reg}}\big(\log (D+G)^{\circ}\big) \to \OO_{(D+G)^{\circ}} \to 0. $$
This induces the exact sequence
$$ 0 \to \Omega^{[1]}_X \to  \Omega^{[1]}_X(\log D) \to \mathcal Q \to 0, $$
where $\mathcal Q$ is a coherent sheaf supported on $D \cup G$ which agrees with $\OO_{(D+G)^{\circ}}$ outside of a finite set.

Second, we need to show that  $\Omega^{[1]}_X\big(\log(D+G)\big) \otimes \OO_X(mL) $ is big for $m \gg 0$. To this end, let $m_1$ be as in the statement of the theorem. Let $\pi\colon \widetilde X \to X$ be a resolution of singularities such that 
$$\widetilde {\mathcal E} = \pi^*\big(\Omega^{[1]}_X\big(\log(D+G)\big) \otimes \OO_X(m_1L)\big) / {\rm torsion}$$
is locally free. Since $\widetilde {\mathcal E}$ is nef, we have $s_3\big(\widetilde {\mathcal E}\big) \geq 0$. Now a simple calculation shows that for $m \gg 0$,
$$ s_3\big(\widetilde {\mathcal E} \otimes \pi^*\OO_X(mL)\big) > 0. $$ 
In other words, if $\tau\colon \PS(\widetilde{ \mathcal E}) \to \widetilde X$ denotes the projection, then for a fixed $m \gg 0$ the line bundle 
$$\mathcal H = \OO_{\mathbb P (\widetilde {\mathcal E})}(1) \otimes \tau^* \pi^*\OO_X(mL) $$
is big and nef. Hence we have
$$ h^0\big(\widetilde X, S^k \big({\widetilde {\mathcal E}} \otimes \pi^*\OO_X(mL)\big)\big) = h^0\big(\mathbb P\big(\widetilde {\mathcal E}\big), \mathcal H^{\otimes k}\big) \sim k^5, $$
and so the sheaf ${\widetilde {\mathcal E}} \otimes \pi^*\OO_X(mL)$ is big. Therefore, pushing forward by $\pi$ we obtain that the sheaf  $\Omega^{[1]}_X\big(\log(D+G)\big) \otimes \OO_X\big((m_1+m)L\big) $ is big for $m \gg 0$. 
\end{proof} 

\begin{pro} \label{pro:vanish1}
Let $X$ be a Calabi-Yau threefold and let $L$ be a nef divisor on $X$ with $\nu(X,L)=2$, such that $\phi\colon X\to Z$ is a good Calabi-Yau model for $L$. Let $S_1, \ldots, S_r$ be all the prime divisors on $X$ orthogonal to $L$, and let $M$ be a nef divisor on $Z$ 
such that $L\sim \phi^*M$. Fix a resolution $\tau\colon Y \to X $ such that the exceptional set of the induced morphism $\pi\colon Y \to Z $ is a simple normal crossings divisor $ E = \sum_{j=1}^s  E_j$ on $Y$. 
\[
\xymatrix{ Y \ar[r]^{\tau} \ar[dr]_{\pi} & X \ar[d]^{\phi} \\
 & Z
}
\]
Under the assumptions of Theorem \ref{amp2}, suppose that $L$ is not semiample. Then we have: 
\begin{enumerate} 
\item[(i)] $H^1\big(Y, \Omega^1_Y(\log E) \otimes \pi^*\OO_Z(mM)\big) = 0 $ for any $m\ll0$,  
\item[(ii)] $H^1\big(Z, \Omega^{[1]}_Z \otimes \OO_Z(mM)\big) = H^3\big(Z, \Omega^{[1]}_Z \otimes \OO_Z(mM)\big)=0$ for any $m\ll0$,
\item[(iii)] $\chi\big(Z,\Omega^{[1]}_Z \otimes \OO_Z(mM)\big) =  0 $ for all integers $m$,
\item[(iv)] $H^q\big(Z,\Omega^{[1]}_Z \otimes \OO_Z(mM)\big) = 0$ for all $q$ and all integers $m$ such that $| m | \gg 0$,
\item[(v)] $ R^1\pi_*\Omega^1_Y(\log E) = 0$,
\item[(vi)] $ H^q\big(Y, \Omega^1_Y(\log E) \otimes \pi^*\OO_Z(mM)\big) = 0$ for all $q$ and all $m \gg 0$. 
\end{enumerate}
\end{pro} 

\begin{proof} 
The statement (i) is \cite[Proposition 2.1]{Wi94}. Since 
\begin{equation}\label{eq:167a}
\pi_*\Omega^p_Y(\log E) \simeq \Omega^{[p]}_Z\quad\text{and}\quad \phi_*\Omega^p_X \simeq \Omega^{[p]}_Z\quad\text{for all }0\leq p\leq 3
\end{equation}
by \cite[Theorem 16.1 and Theorem 1.4]{GKKP11}, the first vanishing of (ii) follows from (i) by the Leray spectral sequence. For the second vanishing, Serre duality, \eqref{eq:167a} and Proposition \ref{pro:2.3new} give
\begin{multline*}
H^3\big(Z, \Omega^{[1]}_Z \otimes \OO_Z(mM)\big) \simeq \Hom\big(\Omega^{[1]}_Z \otimes \OO_Z(mM), \OO_Z\big) \\
\simeq H^0\big(Z,\Omega^{[2]}_Z \otimes \OO_Z(-mM)\big)\simeq H^0\big(X,\Omega^2_X \otimes \OO_X(-mL)\big)=0.
\end{multline*}

For (iii), since $L$ is not semiample, we have that $\chi\big(X, \Omega^1_X \otimes \OO_X(m L)\big) $ is independent of $m$ by Proposition \ref{prop1}(iv). Then \eqref{eq:167a} and the Leray spectral sequence give
\begin{multline}
\chi\big(X, \Omega^1_X \otimes \OO_X(m L)\big)  = \chi\big(Z, \Omega^{[1]}_Z \otimes \OO_Z(m M )\big)\\ \notag
- \chi\big(Z,R^1\phi_*\Omega^1_X \otimes \OO_Z(mM)\big)+ \chi\big(Z,R^2\phi_*\Omega^1_Z \otimes \OO_Z(mM)\big).
\end{multline}
Since $M$ is numerically trivial on the support of $R^q\phi_*\Omega^1_X$ for $q = 1,2$, we obtain that $\chi\big(Z, \Omega^{[1]}_Z \otimes\OO_Z(m M)\big) $ is independent of $m$. Now, if $m\gg0$, then $ \chi\big(Z, \Omega^{[1]}_Z \otimes \OO_Z(m M )\big) \leq 0$ by Propositions \ref{pro:2.3new} and \ref{pro:trivialsurface} and by Theorem \ref{amp2}, whereas when $m\ll0$, then $ \chi\big(Z, \Omega^{[1]}_Z \otimes \OO_Z(m M )\big) \geq 0$ by (ii).

The assertion (iv) is a direct consequence of the proof of (iii).

For (v), by Lemma \ref{lem:vanish1}(b) the support of  $R^1\pi_*\Omega^1_Y(\log E) $ is finite, and thus we need to show that 
\begin{equation}\label{eq:167b}
H^0\big(Z,R^1\pi_*\Omega^1_Y(\log E)\big) = 0.
\end{equation}
Choose a large positive integer $m$. Consider the Leray spectral sequence
\begin{multline*}
E_2^{p,q}=H^p\big(Z,R^q\pi_*\Omega^1_Y(\log E) \otimes \OO_Z(-mM)\big)\\
\Longrightarrow H^{p+q}\big(Y,\Omega^1_Y(\log E) \otimes \pi^*\OO_Z(-mM)\big).
\end{multline*}
Then $E_2^{2,0}=0$ by (iv) and \eqref{eq:167a}, and hence $E_2^{0,1}=0$ by (i), which implies \eqref{eq:167b}.

Finally, (vi) follows from (iv), (v), and from Lemma \ref{lem:vanish1}(a).
\end{proof} 

\begin{proof}[Proof of Theorem \ref{topology}]
Assume that $L$ is not semiample. Fix a resolution $\tau\colon Y \to X $ such that the exceptional set of the induced morphism $\pi\colon Y \to Z $ is a simple normal crossings divisor $ E = \sum_{j=1}^s  E_j$ on $Y$; we may choose $\tau$ as a finite sequence of blowups along smooth centres which lie in the exceptional locus of $\phi$, hence $\pi^*M |_{E_j} \equiv 0$ for all $j$. Then the residue sequence \eqref{eq:residue1} and Proposition \ref{pro:vanish1}(vi) give
$$\chi\big(Y,\Omega^1_Y \otimes \OO_Y(m \pi^*M)\big) = - \sum_{j=1}^s \chi\big(E_j,\OO_{E_j}(m \pi^*M)\big)=-\sum_{j=1}^s \chi(E_j,\OO_{E_j}) $$
for $m \gg 0$. The Riemann-Roch therefore gives
$$\chi\big(Y,\Omega^1_Y \otimes \OO_Y(m \pi^*M)\big) = \chi(Y,\Omega^1_Y)\quad\text{for all }m.$$
By relabelling, we assume that $E_j$ are strict transforms of $S_j$ for $j=1,\dots, r$, that $E_j$ come from blowing-up a smooth curve $B_j$ on some model of $X$ for $j=r+1,\dots, r+t$, and that $E_j$ come from blowing-up points for $j=r+t+1,\dots, s$. Then $ h^{1,1}(Y) = h^{1,1}(X) + s-r$ and 
$$ h^{2,1}(Y) = h^{2,1}(X) + \sum_{j=r+1}^{r+t}g(B_j),$$ 
hence
$$ \chi(X,\Omega^1_X) = s-r-\sum_{j=r+1}^{r+t}g(B_j) -  \sum_{j=1}^s \chi(E_j,\OO_{E_j}). $$
Now $ \chi(E_j,\OO_{E_j}) = 1 -g_j$ for $j=1,\dots,r$, $ \chi(E_j,\OO_{E_j}) = 1 - g(B_j)$ for $j=r+1,\dots, r+t$ and $ \chi(E_j,\OO_{E_j}) = 1$ for $j=r+t+1,\dots,s$, thus
$$\chi(X,\Omega^1_X)  = - r + \sum_{j=1}^r g_j,$$
which together with Proposition \ref{prop1}(iv) finishes the proof of Theorem \ref{topology}. 
\end{proof}

\bibliographystyle{amsalpha}

\bibliography{biblio}

\end{document}